\documentclass[twoside,12pt]{amsart}
\usepackage{amsmath,latexsym,amssymb,mathptm, times,verbatim, enumerate,mathcomp,lscape}
    \usepackage{longtable}\usepackage{tikz,stmaryrd}
\usepackage{pdfpages}
\input amssym.def
\input amssym
\input xypic
\input xy
\xyoption{all}
\usepackage[right=1.5in,left=1.5in,top=1in,bottom=1in]{geometry}

\newcommand{\dbar}[1]{\overline{\overline{#1}}}
\newcommand{\sbar}[1]{\overline{#1}}
\newcommand{\agr}[2][{}]{{{#2}^{\pmb g}_{#1}}}
\def\maxgendeg{\operatorname{maxgd}}
\def\mgd{\maxgendeg} 

\def\HF{\operatorname{HF}}
\def\HS{\operatorname{HS}}
\def\socle{\operatorname{socle}}

\def\ann{\operatorname{ann}}

\def\p{\oplus}

\def\a{\mathfrak a}
\def\b{\mathfrak b}

\def\id{\operatorname{id}}
\def\kk{\pmb k}
\def\m{\mathfrak m}

\def\syz{\operatorname{syz}}
 
\def \w {\wedge}
\def\t{\otimes}
\def\Sym{\operatorname{Sym}}

\def\im{\operatorname{im}}
\def\Tor{\operatorname{Tor}}
\def\Hom{\operatorname{Hom}}
\def\Ext{\operatorname{Ext}}
\def\HH{\operatorname{H}}

\def\n{\mathfrak n}
\def\Fitt{\operatorname{Fitt}}
\def\tsd{\operatorname{topdeg}}
\def\topdeg{\tsd}
\def\M{\mathfrak M}
\def\s{d}

\newcounter{nameOfYourChoice}
\newtheorem{theorem}{Theorem}[section]
\newtheorem{lemma}[theorem]{Lemma}

\newtheorem{proposition}[theorem]{Proposition}

\newtheorem{proposition-no-advance}[equation]{Proposition}
\newtheorem{observation}[theorem]{Observation}

\newtheorem{quick consequences}[theorem]{Quick Consequences}

\newtheorem{hypothesis-no-advance}[equation]{Hypothesis}
\newtheorem{claim-no-advance}[equation]{Claim}
\newtheorem{observation-no-advance}[equation]{Observation}

\theoremstyle{definition}
\newtheorem{facts and definitions}[theorem]{Facts and Definitions}

\newtheorem{definition-no-advance}[equation]{Definition}
\newtheorem{setup}[theorem]{Setup}

\newtheorem{get to work}[theorem]{Get to work}

\newtheorem{remark}[theorem]{Remark}

\newtheorem{remark-no-advance}[equation]{Remark}
\newtheorem{remarks-no-advance}[equation]{Remarks}
\newtheorem{assumption}[theorem]{Assumption}

\newtheorem{careful calculation}[theorem]{Careful Calculation}

\newtheorem{present summary}[theorem]{Present Summary}

\newtheorem{further reductions}[theorem]{Further Reductions}

\newtheorem{chunk}[theorem]{}

\newtheorem{marching orders}[theorem]{Marching Orders}

\newtheorem{circle the wagons}[theorem]{Circle the wagons}

\newtheorem{further notation and quick consequences}[theorem]{Further notation and quick consequences}
\newtheorem{further notation}[theorem]{Further notation}

\newtheorem{chunk-no-advance}[equation]{}

\newtheorem*{Remark}{Remark}

\newtheorem*{proof-of-4a}{Proof of {\rm(\ref{15.2.4a})}}

\newtheorem*{proof outline}{Proof Outline}
\numberwithin{equation}{theorem}

\begin{document}

\title[Totally Reflexive Modules Over Rings
That Are Close To Gorenstein]{Totally Reflexive Modules Over Rings
That Are Close To Gorenstein}

\date{\today}

\author[A.~R.~Kustin]{Andrew R.~Kustin}
\address{Andrew R.~Kustin\\ Department of Mathematics\\ University of South Carolina\\\newline 
Columbia\\ SC 29208\\ U.S.A.} \email{kustin@math.sc.edu}

\author[A.~Vraciu]{Adela~Vraciu}
\address{Adela~Vraciu\\ Department of Mathematics\\ University of South Carolina\\
Columbia\\ SC 29208\\ U.S.A.} \email{vraciu@math.sc.edu}



\thanks{Research partly supported by  Simons Foundation collaboration grant 233597 (ARK) and NSF grant  DMS-1200085 (AV)}

\begin{abstract}Let $S$ be a deeply embedded, equicharacteristic,   Artinian Gorenstein local ring. We prove that if  
$R$ is a non-Gorenstein quotient  of  $S$ of  small colength, then 
every totally reflexive $R$-module is free.
Indeed, the second syzygy of the canonical module of $R$ has a direct summand $T$ which is a test module for freeness over $R$ in the sense that if $\Tor_+^R(T,N)=0$, for some finitely generated $R$-module $N$, then $N$ is free. \end{abstract}

\maketitle

{\small {\em In honor of Craig Huneke, on the occasion of his sixty-fifth birthday.}}

\section{Introduction.}\label{Intro}
Let $R$ be a commutative Noetherian ring.
A finitely generated $R$-module $M$ is called {\it totally reflexive} if there exists a doubly infinite complex of
finitely generated free $R$-modules
$$F:\quad  \cdots \to F_{i+1} \to F_i \to F_{i-1}\to \cdots ,$$
such that $M$ is isomorphic to the cokernel of $F_{i+1}\to F_{i}$, for some $i$, and such that both $F$ and  the dual complex
$\Hom_R (F , R)$ are exact. Totally reflexive modules were introduced by Auslander and Bridger \cite{AB}, who
proved that $R$ is Gorenstein if and only if every finitely generated $R$-module has a finite resolution by totally reflexive $R$-modules. Over a
Gorenstein ring, the totally reflexive modules are precisely the maximal Cohen-Macaulay modules; hence, in particular, every singular Gorenstein local ring has a non-free totally
reflexive module.

Totally reflexive modules have been the object of extensive study ever since their introduction;  yet it remains  an open problem to determine conditions on a ring $R$ that are necessary and sufficient for the existence of a non-free totally reflexive $R$-module. We  follow the lead of Takahashi \cite{T08}  and call a commutative Noetherian local ring
$G$-regular
if every totally reflexive
module over the ring is free. Regular local rings are trivial examples of $G$-regular local rings. Avramov and Martsinkovsky \cite[3.5.(2)]{AM}
 proved that every Golod local ring that is not a hypersurface
 is $G$-regular.

The main result of the paper is Theorem~\ref{Main Result} where we prove that if  
$R$ is a non-Gorenstein quotient of small colength of a deeply embedded, 
equicharacteristic Artinian Gorenstein local ring,  then $R$ is $G$-regular. 
Roughly speaking, a local ring $R$ is deeply embedded if $R=P/I$ for some regular local ring $(P,\M)$ and $I\subseteq \M^{n_0}$ for some large integer $n_0$; see \ref{local}.(\ref{v(R)}) or \cite[4.1]{RS} for a more precise formulation.

We consider local rings $(R,\m)$ of the
form $R = S/(0:_SK)$ where $(S, \n)$ is an Artinian Gorenstein local ring, and
$K$ is an ideal of $S$. Notice that $K$ is also an $R$-module; furthermore, the canonical module of $R$, which is denoted by $\omega_R$, is equal to $K$. 
Sometimes we 
 assume $S = P/\mathfrak a$, where $P$ is a
standard graded polynomial ring and $\a$ is a homogeneous ideal of $P$. 

Our favorite method of showing that a ring $R$ is $G$-regular is to  exhibit a proj-test module  which is a direct summand of a syzygy of the canonical module of $R$. The finitely generated $R$-module $T$ is a proj-test module if the only finitely generated $R$-modules $N$ with $\Tor^R_+(T,N)=0$ are the projective $R$-modules;  see \ref{TMGR}.
We begin by summarizing our
results that establish the existence of direct summands of syzygies of the
canonical module for certain classes of rings. See \ref{con1} for our use of $\Fitt^1$, \ref{beats me} for the proof of Theorem~\ref{(0.1)},
\ref{prove T1} for the proof of Theorem~\ref{0.2A},
 and \ref{proof-of-65.1} for the proof of Theorem~\ref{0.2B}.

\begin{theorem}\label{(0.1)}Let $(R,\m,\kk)$ be a local ring of  the form $R= S/(0:_SK)$ where $(S, \n)$ is an Artinian Gorenstein local ring and
$K$ is an ideal of $S$. If
\begin{align}\label{(1)}\Fitt_S^1(K) &{}= \n,\\\intertext{and}
\label{(2)}\n K :_S \n &{}\ne K :_S \n,\end{align}
then the second syzygy of the canonical module of $R$, viewed as an $R$-module, has a direct summand isomorphic
to $R/\m$.\end{theorem}

 \begin{theorem}\label{0.2A}Let $P$  
be a  commutative Noetherian ring and $A$ and $B$ be ideals of $P$ with $A\subseteq B^2$, and $P/A$ and $P/B$ both Artinian Gorenstein local rings. 
Assume $B$ is generated by a regular $P$-sequence of length at least two. 

Let $R$ be the ring $R=P/(A:_PB)$. Then the second syzygy of the canonical module of $R$, viewed as an $R$-module,  has a direct summand isomorphic
to $B/B^2$. \end{theorem}

\begin{theorem}\label{0.2B} 
 Let $\kk$ be a field of arbitrary characteristic,   $P=\kk[X_1,X_2,X_3,Y_1,\dots,Y_s]$ be a standard-graded polynomial ring over $\kk$  in $3+s$  variables for some nonnegative integer $s$,  $B'$ be an ideal in $P$ which is generated by five linearly independent quadratic forms in the variables $X_1,X_2,X_3$, 
$B''$ be the ideal $(Y_1,\dots,Y_s)$ of $P$,  $B$ be the ideal $B'+B''$ of $P$,
and  $A\subseteq ([P]_1)^5$ be a homogeneous ideal of $P$ with $P/A$ an Artinian Gorenstein local ring.  
Assume    $[\socle(P/B)]_1=0$. 

Let $R$ be the ring $R=P/(A:_PB)$. 
Then the second syzygy of the canonical module of $R$, viewed as an $R$-module,  has a direct summand isomorphic
to $B'/BB'$.
 \end{theorem}

In Section~\ref{small colength} we prove that for each ring under consideration in the main result, Theorem \ref{Main Result}, one of the above Splitting Theorems applies. We prove the Splitting Theorems in Sections~\ref{FST}, \ref{68}, and \ref{65}. In Section~\ref{TM} we prove that the summands produced by the Splitting Theorems are indeed proj-test $R$-modules. We complete the proof by appealing to Observation~\ref{2.13}.

The Splitting Theorems of Sections~\ref{FST}, \ref{68}, and \ref{65} and the Test Module Theorems of Section~\ref{TM} are of independent interest and are stated with more general hypotheses than those of the main result, Theorem~\ref{Main Result}.

\section{Notation, conventions, and preliminary results.}\label{Prelims}

In this paper ${\kk} $ is always a field. If $V$ is vector space over $\kk$, then $\dim_{\kk} V$ is the  dimension of $V$ as a  vector space over $\kk$. 
\begin{chunk}Let $I$ be an ideal in a ring $S$, $N$ be an $S$-module, and  $L$ and $M$ be submodules of $N$. Then 
$$L:_IM=\{x\in I\mid xM\subseteq L\}\quad\text{and}\quad L:_MI=\{m\in M\mid Im
\subseteq L\}
.$$
If $L$ is the zero module, then we also use ``annihilator notation'' to describe these ``colon modules''; that is,
$$\ann_SM=0:_SM
.$$ \end{chunk}

\begin{chunk}\label{con1}If $M$ is a finitely generated module over a local ring $S$, and
$$S^b\xrightarrow{\phi}S^{\mu(M)}\xrightarrow{}M\to 0$$ is a presentation of $M$ with $\mu(M)$ equal to the minimal number of generators of $M$, then we write 
$\Fitt_S^1(M)$ for the ideal of $S$ generated by the entries of $\phi$. Notice that $\Fitt_S^1(M)$ is equal to the usual Fitting ideal $\Fitt_{\mu(M)-1}(M)$ of $M$ and does not depend on the choice of $\phi$. 
\end{chunk}

\begin{chunk} If $M$ is a matrix (or a homomorphism of free $R$-modules), then $I_r(M)$ is the ideal generated by the
$r\times r$ minors of $M$ (or any matrix representation of $M$). We denote the transpose of a matrix $M$ by $M^{\rm T}$. \end{chunk}

\begin{chunk}If $S$ is a ring and $M$ is an $S$-module, then let $\lambda_S(M)$ denote the length of $M$ as an $S$-module. If $R$ is a quotient of $S$, then the colength of $R$ as an $S$-module is $$c_S(R)=\lambda_S(S)-\lambda_S(R).$$\end{chunk}

\begin{chunk}\label{local}``Let $(S,\mathfrak n,\kk)$ be a local ring'' identifies $\mathfrak n$ as the unique maximal ideal of the commutative Noetherian local ring $S$ and $\kk$ as the residue class field $\kk=S/\mathfrak n$. \begin{enumerate}[\rm(a)]
\item The {\it embedding dimension} of $S$ is $\dim_{\kk} (\n/\n^{2})$. 
\item\label{v(R)} The parameter $v(S)$ is defined by $$\textstyle v(S)=\inf\left\{i\left| \dim_{\kk} (\n^i/\n^{i+1}) < \binom{(n-1)+i}{i}\right.\right\},$$where $n$ is the embedding dimension of $S$. (This notation  is introduced in \cite[(4.1.1)]{RS}.) Notice that if $\n$ is minimally generated by $x_1,\dots,x_n$ and $i$ is an integer with $0\le i\le v(S)-1$, then the $\binom{(n-1)+i}i$ monomials of degree $i$ in the symbols $x_1,\dots,x_n$ represent a basis for the $\kk$-vector space $\n^i/\n^{i+1}$.
\item We say that a local ring $S$ is ``deeply embedded'' if $v(S)$ is large.
\end{enumerate}
\end{chunk}
\begin{chunk} \label{2C}
{\bf Gorenstein duality.} If $(S,\n,\kk)$ is an Artinian Gorenstein local ring, then $S$ is an injective $S$-module, $\Hom_S(-,S)$ is an exact functor, and $\lambda_S(M)=\lambda_S(\Hom_S(M,S))$ for all $S$-modules $M$ of finite length. Furthermore, if $A$ is an ideal of $S$, then $(0:_SA)$ is the canonical module of $S/A$. (See, for example, \cite{Ba63,BH}.) The following statements are immediate consequences of the above facts. We refer to them as Gorenstein duality.

\begin{proposition-no-advance} 
If $A$ and $B$ are ideals of  an Artinian Gorenstein local ring $S$, 
then 
\begin{enumerate}[\rm(a)]
\item $\lambda_S(S/A)=\lambda_S (\ann_SA)$,
\item $\ann_S(\ann_S A)=A$,
\item $\ann_S(A+B)=(\ann_SA)\cap (\ann_SB)$,
\item $\ann_S(A\cap B)=(\ann_SA)+ (\ann_SB)$, and
\item if $A\subseteq B$, then $\lambda_S (\frac BA)=\lambda_S(\frac{\ann_SA}{\ann_SB} 
)$.
\end{enumerate}\end{proposition-no-advance}
 \end{chunk}

\begin{chunk}\label{68.1} If $R$ is a commutative ring and $M$ is an $R$-module, then an $R$-module $M'$ is a second syzygy module of $M$ if there 
are projective $R$-modules $P_1$ and $P_2$ and 
 an exact sequence of $R$-modules of the form
$$0\to M'\to P_2 \to P_1 \to M\to 0.$$We denote the second syzygy module of the $R$-module $M$ by $\syz_2^R(M)$. \end{chunk}
\begin{remark-no-advance}Recall from Schanuel's Lemma that the notion of second syzygy module is well-defined up to formation of direct sum with a projective module. In other words, if $M'$ and $M''$ are both second syzygy modules of $M$, then there are projective $R$-modules $P'$ and $P''$ with $M'\p P'$ isomorphic to $M''\p P''$.
\end{remark-no-advance}

\begin{chunk}\label{2.8}As always, let $\kk$ be a field.
\begin{enumerate}[\rm(a)]\item If $M$ is a graded module, then we write $[M]_i$ to represent the component of $M$ which consists of all homogeneous elements of degree $i$. 
\item If $Q$ is a non-negatively graded ring over a local ring $([Q]_0,\m,\kk)$, then $Q$ has a unique maximal homogeneous ideal $\mathfrak M=\m Q+\bigoplus_{1\le i}[Q]_i$. If $M$ is a $Q$-module, then the {\it socle} of $M$ is the vector space $\socle(M)=0:_M\mathfrak M$.

\item\label{2.8.c} Recall that if $M$ is a finitely generated graded module over a graded polynomial ring $P=\kk[x_1,\dots,x_n]$, then the socle degrees of $M$ may be read from the back twists in a minimal homogeneous resolution
$$\textstyle 0\to   C_n=\bigoplus_{i=1}^sP(-\beta_i)\to ...\to C_0$$ 
of $M$ by free $P$-modules. Indeed, the
computation of $\Tor_n^P(M,\kk)$ in each coordinate yields 
 a graded isomorphism $$\socle M\cong \bigoplus_{i=1}^s\kk\big(\big(\sum_{j=1}^n \deg x_j\big )-\beta_i\big).$$
\item  A graded ring $S=\bigoplus_{0\le i}[S]_i$ is called a {\it standard-graded  ${\kk} $-algebra}, if $S$ is a commutative ring with $[S]_0={\kk} $, $S$ is generated as an $[S]_0$-algebra by $[S]_1$, and $[S]_1$ is finitely generated as an $[S]_0$-module.  
\item If $(S,\n,\kk)$ is a local ring, then the associated graded ring of $S$ is the standard-graded $\kk$-algebra
$$\agr{S}=\bigoplus_{i=0}^\infty \frac{\n^i}{\n^{i+1}}.$$

\end{enumerate}
\end{chunk}

\begin{chunk} Let $S$ be a commutative Noetherian  graded ring and $M$ be a  graded $S$-module.
\begin{enumerate}[\rm(a)]\item Define
\begin{align*}
\topdeg (M)&{}=\inf\{\tau\mid [M]_j=0\text{  for all $j$ with $\tau<j$}\},
\notag\\
\mgd(M)&\textstyle{}=\inf \left\{\tau\left| S\left(\bigoplus_{j\le \tau}[M]_j\right)=M\right.\right\}.\notag\end{align*}
The expressions ``$\topdeg$'' and  ``$\mgd$'' are read ``top degree'' and ``maximal generator degree'', respectively.
If $M$ is the zero module, then $\topdeg(M)$ and $\mgd(M)$ are both equal to $-\infty$.
\item If $S$ is non-negatively graded,  $[S]_0$ is Artinian, and $M$ is finitely generated, then 
\begin{enumerate}[\rm(i)]
\item the  {\it Hilbert function} of $M$ is the function $\operatorname{HF}_S(M,\underline{\phantom{x}})$, from the set of integers to the set of non-negative integers, with
$\operatorname{HF}_S(M,i)=\lambda_{[S]_0}([M]_i)$, and  

\item the {\it Hilbert series} of $M$
 is the formal generating function $$\operatorname{HS}_S(M,t)=\sum_{i\in \Bbb Z}\operatorname{HF}_S(M,i)t^i.$$ If $M=S$, we often write $\HS(S,t)$ (or simply $\HS(S)$) in place of $\HS_S(S,t)$.
\end{enumerate}\end{enumerate}
\end{chunk}

\begin{chunk}\label{TMGR}{\bf Test modules and $G$-regularity.} Observation~\ref{2.13} is the starting point for the present paper. The authors of \cite{SV} refer to  \cite[Lemma 3.2]{SV}, which is a similar result,  as ``well-known by the experts''. The present version uses the notion of proj-test module. Related test modules have been used in \cite{M14,CDT,C-SW}.
\begin{definition-no-advance}\label{proj-test}Let $R$ be a commutative Noetherian ring and $T$ be a finitely generated $R$-module. Then   
$T$ is called  a {\it proj-test module}
 for $R$, if the only  finitely generated $R$-modules $N$ with 
$\Tor_{i}^R(T,N)=0$ for all positive $i$ are the  projective $R$-modules. 
\end{definition-no-advance}
\begin{observation-no-advance}\label{Matlis} If $(R,\m,\kk)$ is an Artinian local ring with canonical module $\omega_R$ and $M$ is a finitely generated $R$-module, then $\Ext^i_R(M,R)$ and $\Tor_i^R(M,\omega_R)$ are Matlis duals of one another.\end{observation-no-advance}

\begin{proof}
Recall that   Matlis dual is the functor $(-)^\vee=\Hom_R(-,E)$, where $E$ is the injective envelope of the $R$-module $\kk$; furthermore, in the present situation, $E$ and $\omega_R$ are isomorphic $R$-modules.  
The ring $R$ is complete; so, $N^{\vee\vee}\cong N$ for every finitely generated $R$-module $N$.

Let $F$ be a resolution of $M$ by finitely generated free $R$-modules. Observe that
\begin{align*}
\Ext^i_R(M,R)&{}=\HH^i(\Hom_R(F,R))\\
&{}=\HH^i(\Hom_R(F,\Hom_R(\omega_R,\omega_R)))&&\text{because $R^{\vee\vee}\cong R$}  \\
&{}\cong \HH^i(\Hom_R(F\otimes_R\omega_R,\omega_R))&&\text{by the adjoint isomorphism theorem}\\
&{}=\HH^i((F\otimes_R\omega_R)^\vee)\\
&{}\cong (\HH_i(F\otimes_R\omega_R))^\vee&&\text{because $(-)^\vee$ is exact} 
\\
&{}=(\Tor_i^R(M,\omega_R))^\vee.\end{align*}
\end{proof}

\begin{observation-no-advance}\label{2.13}Let $(R,\m,\kk)$ be an Artinian local ring with canonical module $\omega_R$. If $T$ is a proj-test module for $R$ and $T$ is a summand of a syzygy of $\omega_R$, then $R$ is $G$-regular.  \end{observation-no-advance}

\begin{proof}If  $X$ is  a totally reflexive $R$-module, then, by definition, $\Ext^i_R(X,R)=0$ for all positive $i$ and therefore, from Observation~\ref{Matlis}, $\Tor_i^R(X,\omega_R)=0$ for all positive $i$. It follows that $\Tor_i^R(X,T)=0$ for all positive $i$; and therefore $X$ is a free $R$-module by the definition of proj-test module.\end{proof}
\end{chunk}

\section{Rings of small Gorenstein colength.}\label{small colength}
Let $S$ be an equicharacteristic Artinian Gorenstein local ring. 
Assume that the invariant $v(S)$, of \ref{local}.(\ref{v(R)}), is large.
In this section we prove that if $R$ is a non-Gorenstein quotient of $S$ of small colength, then one of the Splitting Theorems of Section~\ref{Intro} applies to $R$.

\begin{theorem}\label{mar-16-17} Let $(S,\n,\kk)$ be an equicharacteristic  Artinian Gorenstein local ring with embedding dimension at least two, 
 and  $R$ be the ring $R=S/J$ for some proper non-zero 
 ideal   $J$ of $S$. Assume that either
\begin{itemize}
\item $1\le c_{S}(R)\le 4$, or
\item $c_S(R)=5$ and $S$ is a standard-graded algebra over a field $[S]_0$.
\end{itemize}
If the parameter  $v(S)$ 
 is sufficiently large, then one of the Theorems {\rm\ref{(0.1)}, \ref{0.2A}, or \ref{0.2B}} applies to $R$. In particular, the following statements hold, where $k$ denotes the minimal number of generators of $\n/(0:_SJ)$.
\begin{enumerate}[\rm(a)]
\item\label{mar-16-17.a}If $k=0$, then Theorem~{\rm\ref{(0.1)}} applies to $R$.
\item\label{mar-16-17.c}If $k=1$, then   Theorem~{\rm\ref{0.2A}} applies to $R$,  provided  $2c_S(R)\le v(S)$.
\item\label{mar-16-17.b}If $2\le k=c_S(R)-1$, then Theorem~{\rm\ref{(0.1)}} applies to $R$, provided $3\le v(S)$.
\item\label{mar-16-17.d}If $k=2$,  $c_S(R)=4$, and 
$\n^3\subseteq \n(0:_SJ)$, 
then 
 Theorem~{\rm\ref{0.2A}} applies to $R$, provided $5\le v(S)$.
\item\label{mar-16-17.e}If $k=2$,  $c_S(R)=4$, and 
$\n^3\not\subseteq \n(0:_SJ)$,
 then 
 Theorem~{\rm\ref{(0.1)}} applies to $R$, provided $4\le v(S)$.
\item\label{mar-16-17.f}If $k=2$,  $c_S(R)=5$, and $S$ is a standard-graded algebra over a field $[S]_0$, then 
 Theorem~{\rm\ref{(0.1)}} applies to $R$, provided $4\le v(S)$. 
\item\label{mar-16-17.h}If $k=3$,  $c_S(R)=5$,  $3\le \maxgendeg(0:_SJ)$, 
and $S$ is a standard-graded algebra over a field $[S]_0$,
 then 
 Theorem~{\rm\ref{(0.1)}} applies to $R$, provided $3\le v(S)$.
\item\label{mar-16-17.i}If $k=3$,  $c_S(R)=5$, 
$2\le \dim_{\kk}\socle(S/0:_SJ)$, and $S$ is a standard-graded algebra over a field $[S]_0$,
then 
 Theorem~{\rm\ref{(0.1)}} applies to $R$, provided $3\le v(S)$.
\item\label{mar-16-17.g}If $k=3$,  $c_S(R)=5$, 
$\maxgendeg(0:_SJ)\le 2$, 
 $\dim_{\kk}\socle(S/(0:_SJ))=1$,  and $S$ is a standard-graded algebra over a field $[S]_0$, then 
 Theorem~{\rm\ref{0.2B}}  applies  to $R$, provided $5\le v(S)$.
 \end{enumerate}
\end{theorem}

Assertions (\ref{mar-16-17.a}) -- (\ref{mar-16-17.e}) of Theorem~\ref{mar-16-17} are proven in \ref{Proof of Theorem}; assertions (\ref{mar-16-17.f}) -- (\ref{mar-16-17.g}) are proven in \ref{Proof of Theorem -- part 2}. 
First we make some remarks about the statement and introduce some notation.

\begin{remark-no-advance}\label{83.1.1}The ring $S$ of Theorem~\ref{mar-16-17}. is equicharacteristic and complete; so, the Cohen structure theorem guarantees that $S$ contains a copy of the residue field $\kk$; that is, there is a commutative diagram $$\xymatrix{\kk\ar@{^(->}[rr]\ar[dr]_(.45){\cong}&&S\ar[dl]^(.45){\text{natural map}}\\&S/\n.}$$ Let $n$ be the embedding dimension of $S$,  
 $U$ be an $n$-dimensional vector space over $\kk$, $P$ be the polynomial ring $\Sym_{\bullet}^{\kk}(U)$, and   $\M$ be the maximal  ideal of $P$ generated by $\Sym_1 U$. 
Fix a vector space isomorphism $U\xrightarrow{\cong}\n/\n^2$.
The ring $S$ is Artinian; so, the 
inclusion $\xymatrix{\kk\ar@{^(->}[r]&S}$
and the isomorphism $U\cong \n/\n^2$ combine to induce a surjection $\xymatrix{\pi:\ P\ar@{->>}[r]&S}$. Let $A$ be the kernel of $\pi$. No harm is done if we write $$P/A=S.$$ 
Recall from \ref{local}.(\ref{v(R)}) that $A\subseteq \M^{v(S)}$; but $A\not\subseteq \M^{v(S)+1}$. \end{remark-no-advance}

\begin{remark-no-advance}\label{apr14} In the situation of of Theorem~\ref{mar-16-17}, let $K$ be the ideal $(0:_SJ)$ of $S$. Gorenstein duality (see \ref{2C}) guarantees that $$(0:_SK)=J.$$ Observe that
\begin{enumerate}[\rm(a)]
\item\label{83.1.a} $K$ is an $R$-module,
\item\label{83.1.b} $K$ is the canonical module of $R$, and
\item\label{83.1.c} $\lambda_S(K)=c_S(R)$.
\end{enumerate}
Assertion (\ref{83.1.a}) is obvious. For (\ref{83.1.b}), observe that  
$S$ and $R$ have the same Krull dimension 
 and $S$ is a Gorenstein local ring which maps onto $R$. It follows that the canonical module of $R$ is
$$
\textstyle\omega_R=\Hom_S(R,S)=\Hom_S(S/J,S)=(0:_SJ)=K.$$Assertion (\ref{83.1.c}) holds because
$\lambda_S(R)=\lambda_S(\omega_R)=\lambda_S(K)=\lambda_S(S)-\lambda_S(S/K)$;
hence, $$\lambda_S(S/K)=\lambda_S(S)-\lambda_S(R)=c_S(R).$$ 
\end{remark-no-advance}

\begin{remark-no-advance}\label{B1} In the situation of of Theorem~\ref{mar-16-17}, let $B=\pi^{-1}(K)$ be the preimage of $K$ in $P$. It follows that $$R=\frac SJ=\frac{S}{(0:_SK)}= \frac{P/A}{(A:_PB)/A}=\frac P{(A:_PB)}.$$
\end{remark-no-advance}

\begin{remark-no-advance} 
\label{R3.2.a} In the situation of of Theorem~\ref{mar-16-17}, the 
 inequality $
k+1\le c_S(R)$ always holds because 
$$\textstyle k=\dim_{\kk}\left(
\frac{\n}{K+\n^2}\right)
\le \lambda_S(\n/K)=\lambda_S(S/K)-1=c_S(R)-1.$$The final equality is established in  Remark~\ref{apr14}.(\ref{83.1.c}).
\end{remark-no-advance}

\begin{remark-no-advance}\label{R3.2.b} The assertion of Theorem~\ref{mar-16-17} does not hold if the ring $R$ is Gorenstein; and for this reason the hypotheses of Theorem~\ref{mar-16-17} exclude $J=0$, exclude $c_S(R)=0$, and exclude the case where $\n$ is a principal ideal. Indeed, if $J$ were equal to $0$ or if $c_S(R)$ were equal to $0$, then $R$ would be equal the Gorenstein ring $S$. Similarly, if $\n$ were a principal ideal, then $S$ and $R$ would both be Gorenstein rings of  the form $\kk[X_1]/(X_1^{N_i})$ for  integers $N_1$ and $N_2$.\end{remark-no-advance}

\begin{remark-no-advance}\label{R3.2.c}The rings of Theorem~\ref{mar-16-17}.(\ref{mar-16-17.a}) are called Teter rings. It was already shown in \cite{SV} that Teter rings are $G$-regular. In fact, the present paper is inspired by \cite{SV}. 
\end{remark-no-advance}

\begin{remark-no-advance}\label{R3.2.d}All cases with  $1\le c_S(R)\le 5$ are covered in (\ref{mar-16-17.a})--(\ref{mar-16-17.h}) of Theorem~\ref{mar-16-17} because once (\ref{mar-16-17.a}), (\ref{mar-16-17.c}), and (\ref{mar-16-17.b}) are established, then, it follows from Remark~\ref{R3.2.a}, that it is  only necessary to consider $2\le k\le c_S(R)-2$ for $c_S(R)$ equal to $4$ or $5$. 
\end{remark-no-advance}

\begin{chunk}\label{Proof of Theorem}
{\bf Proof of assertions (\ref{mar-16-17.a}) -- (\ref{mar-16-17.e}) of Theorem~\ref{mar-16-17}.} Retain the notation $P$, $U$, $\M$, $A$, $K$, and $B$ which is introduced in Remarks~\ref{83.1.1}, \ref{apr14}, and \ref{B1}. 
Choose a minimal generating set $x_1,\dots,x_k,y_1,\dots,y_s$ for $\mathfrak n$ such that $x_1,\dots,x_k$ represents a minimal generating set for $\mathfrak n/K$ and $y_1,\dots,y_s$ are in $K$. Select $Y_1,\allowbreak \dots,\allowbreak Y_s,\allowbreak X_1,\allowbreak \dots,\allowbreak X_k\in U$ to be preimages of $y_1,\dots,y_s,x_1,\dots,x_k$, respectively. Observe that $Y_1,\dots,Y_s,X_1,\dots,X_k$ is a basis for the $\kk$ vector space $U$; indeed,
$P$ is the 
polynomial ring $$P=\kk[Y_1,\dots,Y_s,X_1,\dots,X_k].$$

  Let $(x)$ and $(y)$  denote the ideals  $(x_1,\dots, x_k)$ and $(y_1,\dots, y_s)$ of $S$, respectively, and let  $d$ denote the top degree of the associated graded ring $$\agr{(S/K)}=\bigoplus_{i=0}^\infty \frac{\n^i+K}{\n^{i+1}+K}$$ of $S/K$.  
In other words, $d$ is the smallest positive
integer with $(x)^{d+1} \subseteq K$. 
There are strict inclusions
\begin{equation}\label{(4)} K \subsetneq K + (x)^{d} \subsetneq \dots \subsetneq K + (x)^2 \subsetneq K + (x) = \n\subseteq S.
\end{equation}
 Notice that each quotient of consecutive terms
in (\ref{(4)}) is annihilated  by $\n$ and 
the Hilbert series of $\agr{(S/K)}$ is
$${\textstyle \HS\left(\agr{(S/K)},t\right)}=\sum_{i=0}^d\dim_{\kk}\frac{K + (x)^i}{K + (x)^{i+1}} t^i;$$
furthermore, 
\begin{align*}k&{}=\dim_{\kk}\frac{K+(x)}{K+(x)^2}= \text{the coefficient of $t$ in
$\HS(\agr{(S/K)},t)$}\quad\text{and}\\
c_S(R)&{}=\lambda_S(S/K)=\sum_{i=0}^d\dim_{\kk}\frac{K + (x)^i}{K + (x)^{i+1}}
=\HS(\agr{(S/K)},1).
\end{align*}

\medskip\noindent {\bf(\ref{mar-16-17.a})} The parameter $k$ is equal to $0$; consequently, $\n=K$ and $y_1,\dots,y_s$ is a minimal generating set for $\n=K$. 
The  hypotheses that $\n$ is not principal and $K$ is not zero (that is, $J$ is a proper ideal of $S$) guarantee that $2\le s$.  
Therefore, each $y_i$ appears in a Koszul relation on $y_1,\dots,y_s$. It follows that $$\n\subseteq \Fitt_S^1(\n)=\Fitt_S^1(K)\subseteq \n;$$ and therefore condition (\ref{(1)}) of Theorem~\ref{(0.1)} is satisfied. Furthermore, condition (\ref{(2)}) is satisfied because $1\in (K:\n)$, but $1\not \in (\n K:\n)$.
We have shown that if $k=0$, then Theorem~\ref{(0.1)} applies to the ring $R$.

\medskip\noindent{\bf(\ref{mar-16-17.c})} 
We assume that $k=1$ and $2c_S(R)\le  v(S)$. It follows that each quotient of consecutive terms in (\ref{(4)}) has
length one;  thus, $d+1 = c_S(R)$ and  \begin{equation}\label{eq801}(y_1,\dots, y_s, x_1^{d+1})\subseteq K.\end{equation} On the
other hand, $$\lambda_S(S/(y_1,\dots, y_s, x_1^{d+1})) = d+1=c_S(R)=\lambda_S(S/K);$$ and equality holds in (\ref{eq801}). 
 In this case, the ideal $B$ of $P$ from (\ref{B1}) is generated by a regular sequence. The regular sequence has length at least two because $\n$ is not zero and not principal. Furthermore, the hypothesis that $$2c_S(R)\le v(S)$$  forces 
$$A\subseteq \M^{v(S)}\subseteq \M^{2c_S(R)}=\M^{2(d+1)}\subseteq B^2.$$ All of the hypotheses of Theorem~\ref{0.2A} are satisfied by the ring $R$.

\medskip\noindent{\bf(\ref{mar-16-17.b})} We assume that $2\le k= c_S(R)-1$ and $3\le v(S)$. 
 It follows from (\ref{(4)}) that $K + (x)^2 = K$; hence, 
$$(y)+(x)^2\subseteq K.$$ 
In fact,  equality holds, because 
$$
\frac{\n}{(y)+(x)^2}
\quad\text{and}\quad \frac \n K$$
both have length $k$.

The hypothesis  $3\le v(S)$ ensures that $K=(y)+(x)^2$ is minimally generated by $y_1,\dots,y_s$, together with the $\binom{k+1}2$ monomials of degree two in $x_1,\dots,x_k$. 

\begin{chunk-no-advance}\label{any}
The ideal $K$ has at least two minimal generators; so, 
any Koszul relation involving any $y_i$ (with $1\le i\le s$) and any other minimal generator exhibits $y_i$ as an element of $\Fitt_S^1(K)$. \end{chunk-no-advance}

\noindent Furthermore, the parameter $k$ is at least two; so, if $i$ and $j$ are arbitrary with $$1\le i\neq j\le k,$$ then the relation $x_i(x_j^2)-x_j(x_ix_j)$ exhibits as $x_i$ and $x_j$ elements of $\Fitt^1_SK$.  
It is now apparent that hypothesis (\ref{(1)}) of Theorem~\ref{(0.1)} holds.  It is clear that $x_1\n\subseteq K$. The hypothesis  $3\le v(S)$ also ensures   that $x_1\n\not\subseteq \n K$. 
 We conclude that both  assumptions of 
Theorem \ref{(0.1)} hold in this case.

\begin{chunk-no-advance}\label{mar-18}
{\bf  The case $(k,c_S(R))=(2,4)$.} We make some preliminary calculations 
 before dividing this case into the two sub-cases (\ref{mar-16-17.d}) and (\ref{mar-16-17.e}). 

Examine the filtration  (\ref{(4)}) in order to see that $\HS(\agr{(S/K)})=1+2t+t^2$.
It follows, in particular, that  \begin{equation}\label{one-dim}(K+(x_1, x_2)^2)/K\end{equation} is a one-dimensional $\kk$-vector space. It also follows that $(x_1,x_2)^3\subseteq K$. 
The hypothesis $3\le v(S)$ ensures that the
$\kk$-submodule of $S$ which is spanned by $x_1^2$, $x_1x_2$, $x_2^2$ is a three-dimensional vector space, which we call $V$.
Select a basis $q_0,q_1,q_2$ for $V$  so that $q_0$ represents a basis for (\ref{one-dim}) 
 and $q_1$ and $q_2$ are in $K$. It follows that 
\begin{equation}\label{(5)} (y_1,\dots, y_s,q_1,q_2) + (x_1, x_2)^3 \subseteq K;\end{equation} and therefore,
\begin{equation}\label{3.2.3}3=\lambda_S \left(\frac{\n}{K}\right) \le \lambda_S
\left(\frac{\n}{
(y_1,\dots, y_s,q_1 ,q_2 ) + (x_1, x_2)^3}
\right).
\end{equation}On the other hand, the module on the right side of (\ref{3.2.3}) has length at most three because this module is generated by the images of 
$q_0$, $x_1$, and $x_2$, and these generators give rise to a filtration of length three whose factors are vector spaces of dimension at most one. 
We conclude that
equality holds in both (\ref{3.2.3}) and  
(\ref{(5)}); in particular,
 \begin{equation}\label{(5')} (y_1,\dots, y_s,q_1,q_2) + (x_1, x_2)^3 = K.\end{equation}

The hypothesis $4\le v(S)$ guarantees that the elements $y_1,\dots,y_s,q_1,q_2$ are the beginning of a minimal generating set for $K$; furthermore,
\begin{equation}\label{iff}(y)+(q_1,q_2)=K\iff \n^3\subseteq \n K.\end{equation}
 We consider two sub-cases.\end{chunk-no-advance}

\noindent{\bf(\ref{mar-16-17.d})} Assume $(k,c_S(R))=(2,4)$, 
$\n^3\subseteq \n K$,
 and  $5\le v(S)$. Continue the discussion of \ref{mar-18}.
It follows from 
  (\ref{(5')}) and (\ref{iff}) 
that $$(x_1, x_2)^3 \subseteq (y_1,\dots, y_s,q_1 ,q_2 ).$$ 

Recall the polynomial ring $P=\Sym_\bullet^{\kk}U$ with $P/A=S$ which was introduced Remark~\ref{83.1.1}. Select $Y_1,\dots,Y_s,X_1,X_2\in U$ to be preimages of $y_1,\dots,y_s,x_1,x_2$, respectively. Observe that $Y_1,\dots,Y_s,X_1,X_2$ is a basis for the $\kk$ vector space $U$. Select $Q_1$ and $Q_2$ in the vector space $\kk X_1^2\p\kk X_1X_2\p\kk X_2^2$
 to be preimages of $q_1$ and $q_2$, respectively. 
Notice that  $B$, from Remark~\ref{B1},  is equal to $(Y_1,\dots,Y_s,Q_1,Q_2)+A$.

Observe that $Q_1,Q_2$ is a regular sequence in $P$. Otherwise, $Q_1$ and $Q_2$ have a common factor and the
$\kk$-submodule of $P$ which is 
 spanned by $$\{X_iQ_j\mid 1\le i,j\le 2\}$$
is a vector space (which we call $V'$) of dimension at most $3$. 
On the other hand, the image of the above $V'$ in $S$
contains the four dimensional vector space 
spanned by the monomials of degree $3$ in $x_1,x_2$. 
(Keep in mind that the monomials in $y_1, \allowbreak \dots,\allowbreak y_s,\allowbreak x_1,\allowbreak x_2$ of degree $i$ represent a basis for 
$\n^i/\n^{i+1}$ for $0\le i\le v(S)-1$.) This contradiction establishes the claim that $Q_1,Q_2$ is a regular sequence in $P$.

Now that we know that $Q_1,Q_2$ is a regular sequence, a quick calculation shows that 
$$\M^{5} \subseteq (Y_1,\dots,Y_s,Q_1,Q_2)^{2}.$$
Therefore the hypothesis that $5\le v(S)$ guarantees that
$$A\subseteq \M^{v(S)} \subseteq \M^{5} \subseteq (Y_1,\dots,Y_s,Q_1,Q_2)^{2}.$$ Thus, $B$ is generated by the regular sequence $Y_1,\dots,Y_s,Q_1,Q_2$; this regular sequence has length at least two; and $A\subseteq B^2$.  All of the hypotheses of Theorem~\ref{0.2A} are in effect.

\medskip\noindent{\bf(\ref{mar-16-17.e})}  Assume $(k,c_S(R))=(2,4)$,  
$\n^3 \not\subseteq \n K$,
 and $4\le v(S)$.    
Continue the discussion of \ref{mar-18}. It follows from  
(\ref{(5')}) and (\ref{iff}) 
 that \begin{equation}\label{AH}(x_1, x_2)^3 \not\subseteq (y_1,\dots, y_s,q_1 ,q_2 ).\end{equation}
We saw in the proof of (\ref{mar-16-17.d}) that the quadratic forms $Q_1,Q_2$ of  $\kk[X_1,X_2]$ must have a common linear factor. In other words, there are homogeneous linear forms $L, L_1, L_2$ in $\kk[X_1,X_2]$ with  $Q_1 = LL_1$ and $Q_2 = LL_2$.
If $\ell_1$ and $\ell_2$ are the images in $S$ of $L_1$ and $L_2$, then
$\ell_1q_2 = \ell_2q_1$ 
is a relation on a minimal generating set for $K$ that exhibits $\ell_1$ and $\ell_2$ as elements of $\Fitt_S^1(K)$.  
The linear forms  $L_1,L_2$ must be linearly independent in $P$; therefore, 
$\kk(\ell_1,\ell_2) = \kk(x_1, x_2)$ as sub-vector spaces of $S$; hence, $x_1$ and $x_2$ are in $\Fitt_S^1(K)$. Of course, the argument of (\ref{any}) shows that all of the elements  
 $y_1,\dots, y_s$  are in $\Fitt_S^1(K)$.  
This
verifies that Hypothesis (\ref{(1)}) in Theorem \ref{(0.1)} holds for $R$. 
We verify Hypothesis
(\ref{(2)}) by showing that $q_0\in K:_S\n$ but $q_0\not\in \n K:_S\n$. The first assertion holds because
$$q_0\n\subseteq (y_1,\dots,y_s)+(x_1,x_2)^3\subseteq K.$$ For the second assertion,  write $(x_1,x_2)^3=(q_0,q_1,q_2)(x_1,x_2)$. The 
 ambient hypothesis $\n^3\not\subseteq \n K$ guarantees that 
there is an element $\ell\in \kk x_1\p \kk x_2$ with 
$\ell q_0\in K\not \in \n K$.
\qed
\end{chunk}

Assertions (\ref{mar-16-17.f}) -- (\ref{mar-16-17.g}) of Theorem~\ref{mar-16-17} are proven in \ref{Proof of Theorem -- part 2}. 
We first  prove some preliminary results.

\begin{lemma}\label{71.5} Let $B'$ be a homogeneous ideal of the standard-graded polynomial ring $P'=\kk[X_1,X_2]$. 
If 
\begin{enumerate}[\rm(a)]\item\label{71.5.a}
the embedding dimension of $P'/B'$ is 2 and the length of $P'/B'$ is $5$, or
\item\label{71.5.b} the embedding dimension of $P'/B'$ is 2,  the length of $P'/B'$ is $4$, and
the maximal generator degree of $B'$ is at least $3$,\end{enumerate} then $\Fitt^1_{P'}(B')=(X_1,X_2)$.\end{lemma}
\begin{proof} (\ref{71.5.a}) The Hilbert series of $P'/B'$ is either $1+2t+2t^2$ or $1+2t+t^2+t^3$. We first assume that $\HS(P'/B')=1+2t+2t^2$. Compare $\HS(P'/B')$ and $\HS(P')$ to see that $\HS(B')=t^2+4t^3+\dots$~. It follows that $B'$ is minimally generated by one quadratic form and two cubic forms. Furthermore, we observe that $$\socle(P'/B')=[P'/B']_2\cong \kk(-2)^2.$$ The ideal $B'$ of $P'$ is perfect of grade two and the  minimal homogeneous resolution of $P'/B'$ by free $P'$-modules looks like
\begin{equation}\notag 0\to P'(-4)^2\xrightarrow{\ \ \phi \ \ } \begin{matrix}P'(-2)^1\\\p\\ P'(-3)^2\end{matrix} \to P'.\end{equation} (We used \ref{2.8}.(\ref{2.8.c}).) The matrix $\phi$ has the form \begin{equation}\label{form}\phi=\bmatrix \phi_{1,1}&\phi_{1,2} \\\phi_{2,1}&\phi_{2,2}\\\phi_{3,1}&\phi_{3,2}\endbmatrix,\end{equation} where each $\phi_{i,j}$ is a homogeneous form in $P'$ and $$\deg \phi_{i,j}=\begin{cases} 2&\text{if $i=1$}\\1&\text{if $2\le i$}.\end{cases}$$
The Hilbert-Burch Theorem guarantees that $B'$ is generated by the maximal order minors of $\phi$; hence $B'$ is contained in the ideal generated by the degree one entries of $\phi$. The ideal $B'$ has grade $2$; thus,  the degree one entries of $\phi$ generate the entire ideal $(X_1,X_2)$.

We now assume that $\HS(P'/B')=1+2t+t^2+t^3$. In this case, 
the Hilbert function of $B'$ is $2t^2+3t^3+5t^4+\cdots$\ .
The ideal $B'$ has two minimal quadratic generators; but these generators have a linear factor $L$ in common; consequently,
the Hilbert function of the ideal generated by the two quadratics is $2t^2+3t^3+4t^4+\dots$.  We conclude that $B'$ is minimally generated by two quadratic forms and a homogeneous form of degree $4$. The element of $P'/B'$ represented by $L$ is in the socle of $P'/B'$ and $[P'/B']_3$ is contained in the socle of $P'/B'$. At this point we know that the  minimal homogeneous resolution of $P'/B'$ by free $P'$-modules looks like
$$0\to \begin{matrix}P'(-3)\\\p\\ P'(-5)\\\p\\ F  
\end{matrix}
\xrightarrow{\ \ \phi \ \ } \begin{matrix}P'(-2)^2\\\p\\ P'(-4)^1\end{matrix} \to P',$$for some homogeneous free $P'$-module $F$. Rank considerations show that $F$ is zero. The matrix $\phi$ has the form (\ref{form}),  
where each $\phi_{i,j}$ is a homogeneous form in $P'$, $\phi_{3,1}=0$ and $$\deg \phi_{i,j}=\begin{cases} 1&\text{if $(i,j)$ equals $(1,1)$, $(2,1)$, or $(3,2)$}\\3&\text{if  $(i,j)$ equals $(1,2)$ or $(2,2)$}.\end{cases}$$
 Once again, the Hilbert-Burch Theorem guarantees that $B'$ is generated by the maximal order minors of $\phi$; hence $B'\subseteq (\phi_{1,1},\phi_{2,1})$.  The ideal $B'$ still has grade $2$; thus,  $(\phi_{1,1},\phi_{2,1})=(X_1,X_2)$, and the proof is complete.

\medskip \noindent (\ref{71.5.b}) The ring $P'/B'$ is standard-graded; consequently,  $$[P'/B']_i=0\implies [P'/B']_{i+1}=0.$$ It follows that $\HS(P'/B')=1+2t+t^2$. The ideal $B'$ has two minimal generators of degree two and, by hypothesis, at least one minimal generator of degree at least three. The quadratic generators must have a linear factor $L$ in common as in the $\HS=1+2t+t^2+t^3$ case. It follows follows that $B'$ has a minimal cubic generator and no further generators are needed. The element $L$ represents an element of degree one  in the socle of $P'/B'$ and $[P'/B']_2$ is also contained in the socle. Rank considerations show that the socle can not be any larger than $\kk(-1)\oplus \kk(-2)$ as in the $\HS=1+2t+t^2+t^3$ case. Thus, the  minimal homogeneous resolution of $P'/B'$ by free $P'$-modules looks like
$$0\to \begin{matrix}P'(-3)\\\p\\ P'(-4)  
\end{matrix}
\xrightarrow{\ \ \phi \ \ } \begin{matrix}P'(-2)^2\\\p\\ P'(-3)^1\end{matrix} \to P',$$for some homogeneous free $P'$-module $F$.  The matrix $\phi$ has the form (\ref{form}),  
where each $\phi_{i,j}$ is a homogeneous form in $P'$, $\phi_{3,1}=0$ and $$\deg \phi_{i,j}=\begin{cases} 1&\text{if $(i,j)$ equals $(1,1)$, $(2,1)$, or $(3,2)$}\\2&\text{if  $(i,j)$ equals $(1,2)$ or $(2,2)$}.\end{cases}$$Once again, the entries in the first column of $\phi$ must generate $(X_1,X_2)$. 
\end{proof}

\begin{lemma}\label{71.8} Let $P'$ be a standard-graded polynomial ring over the field $\kk$ with maximal homogeneous ideal $\mathfrak M'$ and let $B'$ be an $\mathfrak M'$-primary homogeneous ideal of  $P'$. If the top degree of $P'/B'$ is less than the maximal generator degree of $B'$, then 
$$(B':_{P'}\mathfrak M')\neq (\mathfrak M' B':_{P'} \mathfrak M').$$
\end{lemma}
\begin{proof}Let $d$ denote the top degree of $P'/B'$.
First notice  that $\mgd(B')$ is equal to $d+1$. Indeed, the inequality 
$d+1\le \mgd(B')$ is part of the hypothesis. On the other hand, $[P']_{d+1}=[\mathfrak M']_1[P']_{d}\subseteq B'$. (The equality holds because $P'$ is standard-graded.) It follows that if $p\in [P']_{d+\ell}$, for some $\ell$ with $2\le \ell$, then $p\in \mathfrak M' B'$ (again because $P'$ is standard graded) and $p$ is not a minimal generator of $B'$; in other words, $\mgd(B')\le d+1$.  

Now observe that $$[B']_{d}[\mathfrak M']_1\subsetneqq [B']_{d+1} =[\mathfrak M']_{d+1}=[\mathfrak M']_{d}[\mathfrak M']_{1}.$$ The strict inclusion occurs because $B'$  requires a minimal generator in degree ${d+1}$. It follows that there exists an element $\theta$ of $[\mathfrak M']_{d}$ with $\theta  [\mathfrak M']_{1}\not\subseteq [B']_{d}[\mathfrak M']_1$. This $\theta$ is in $(B':_{P'}\mathfrak M')\setminus (\mathfrak M' B':_{P'} \mathfrak M')$.
\end{proof}

\begin{lemma}\label{71.9} Let $P$ be a standard-graded polynomial ring over the field $\kk$ with maximal homogeneous ideal $\mathfrak M$, and  let $A\subseteq B$ be $\mathfrak M$-primary  
ideals of $P$, with $P/A$ a Gorenstein ring. Denote $P/A$ by $S$, $B/A$ by $K$, and $S/(0:_SK)$ by $R$. Let $Y_1,\dots,Y_s,X_1,\dots,X_k$ be a basis for $[P]_1$ and  $P'$ be the subring $\kk[X_1,\dots,X_k]$ of $P$. Denote the maximal homogeneous ideal of $P'$ by $\mathfrak M'$. Suppose that $B=(Y_1,\dots,Y_s,B'P)$ for some ideal $B'$ of $P'$. If 
\begin{enumerate}[\rm(a)]\item\label{71.9.a}$\Fitt^1_{P'}B'=\mathfrak M'$, \item\label{71.9.b}$\mgd(B')<\tsd (P'/B')$, and \item\label{71.9.c}$\tsd (P'/B')+1\le v(S)$,\end{enumerate} then  Theorem~{\rm\ref{(0.1)}} applies to $R$. \end{lemma}

\begin{proof}Let $d$ denote the top degree of $P'/B'$ and $\n$ denote the maximal homogeneous ideal of $S$. Hypothesis (\ref{71.9.b}) enables us to apply Lemma~\ref{71.8} in order  to learn that there is an element $\theta'\in [P']_{d}$ with $\theta'\mathfrak M'\subseteq B'$ but $\theta'[\mathfrak M']_{1}\not\subseteq [\mathfrak M' B']_{d+1}$. Observe that $\theta'$ is also in $[P]_{d}$ with $$\theta'\mathfrak M\subseteq B\quad\text{but}\quad \theta'[\mathfrak M]_{1}\not\subseteq [\mathfrak M B]_{d+1}.$$ Apply the natural quotient map $\pi:P\to S$. Let $\theta\in [S]_{d}$ denote  $\pi(\theta')$. Hypothesis (\ref{71.9.c}) guarantees that $$\pi:[\mathfrak M B]_{d+1}\to [\mathfrak n K]_{d+1}$$ is an isomorphism; thus $\theta\in [S]_{d}$ with $$\theta\mathfrak n\subseteq K\quad\text{but}\quad \theta[\mathfrak n]_{1}\not\subseteq [\mathfrak n K]_{d+1},$$
and Hypothesis (\ref{(2)}) of Theorem~\ref{(0.1)} holds for $R$.

Now we verify that Hypothesis  (\ref{(1)}) of Theorem~\ref{(0.1)} holds for $R$.
Let $f_1,\dots, f_{c_1}$ be a minimal homogeneous generating for $B'$ and 
$$(P')^{c_2}\xrightarrow{\phi'} (P')^{c_1}\xrightarrow{\pmb f=\bmatrix f_1,\dots, f_{c_1}\endbmatrix } B'$$ be a presentation of $B'$. It follows that
$$P^{c_2+sc_1+\binom s2}\xrightarrow{\phi} P^{c_1+s}\xrightarrow{\bmatrix \pmb f&Y_1&\dots& Y_{s}\endbmatrix } B,$$with
$$\phi=\left[\begin{array}{c|c|c|c|c|c|c|c}
\phi'&Y_1 I_{c_1}&Y_2 I_{c_1}&\cdots&Y_s I_{c_1}&&&\\\hline
&-\pmb f&&&&-Y_2&&\\\hline
&&-\pmb f&&&Y_1&&\\\hline
&& &\ddots&&&\cdots &\\\hline
&& &&&&&-Y_s\\\hline
&&&&-\pmb f&&&Y_{s-1}\\
\end{array}\right],$$
is a presentation $B$ by free $P$-modules with $f_1,\dots, f_{c_1}, Y_1,\dots, Y_{s}$
a minimal generating set for $B$. Hypothesis (\ref{71.9.a}) guarantees that $I_1(\phi')=\mathfrak M'$. It follows that \begin{equation}\label{71.10.1}I_1(\phi)=\mathfrak M.\end{equation} Apply $-\t_P S$ to obtain the complex 
\begin{equation}\label{cx}S^{c_2+sc_1}\xrightarrow{\phi\t_P S } S^{c_1+s}\xrightarrow{\bmatrix \pmb f&Y_1&\dots& Y_{s}\endbmatrix\t_PS } K.\end{equation} Observe that $\pi_1(f_1),\dots, \pi(f_{c_1}),\pi(Y_1),\dots,\pi(Y_s)$ is still a minimal generating set for $K$ because Hypothesis (\ref{71.9.c}) ensures that $v(S)$ is larger than the maximal generator degree of $K$. (See the proof of Lemma~\ref{71.8}, if necessary.) The fact that (\ref{cx}) is a complex ensures that $I_1(\pi(\phi))\subseteq \Fitt^1_S(K)$. On the other hand, (\ref{71.10.1}) ensures that
$I_1(\pi(\phi))$, which is equal to $\pi(I_1(\phi))$, is equal to $\pi(\mathfrak M)=\mathfrak n$. Thus, $\Fitt^1_S(K)=\n$ and Hypothesis  (\ref{(1)}) of Theorem~\ref{(0.1)} holds for $R$.
\end{proof}

\begin{chunk}\label{Proof of Theorem -- part 2}
{\bf Proof of assertions (\ref{mar-16-17.f}) -- (\ref{mar-16-17.g}) of Theorem~\ref{mar-16-17}.} Throughout this proof, the ring $S$ of Theorem~\ref{mar-16-17} is standard-graded over $[S]_0=\kk$, the polynomial ring $P$ of Remark~{83.1.1} is standard-graded over $[P]_0=\kk$, the ideals $K$ in $S$ of Remark~\ref{apr14} and $A$ and $B$ in $P$ of Remarks~\ref{83.1.1} and \ref{B1} are homogeneous.
The ring $S/K$ is standard-graded and is equal to the associated graded ring $\agr{(S/K)}$.

Recall the elements $x_1,\dots,x_k,y_1,\dots, y_s$ in $S$ and the corresponding elements $X_1,\allowbreak \dots,\allowbreak X_k,\allowbreak Y_1\dots,\allowbreak  Y_s$ of P from the beginning of \ref{Proof of Theorem}. Let $P'$ be the subring 
$$P'=\kk[X_1\dots,X_k]$$ of $P$. Observe that $B=(Y_1,\dots,Y_s)+B'P$ for some  homogeneous ideal $B'$ of $P'$; furthermore, there is a natural 
isomorphism $(P'/B')\cong S/K$, which is induced by
\begin{equation}\label{P'/B'}\xymatrix{P'\ar@{^(->}[r]&P\ar@{->>}[r]^{\pi}& S.}\end{equation}

\medskip\noindent {\bf (\ref{mar-16-17.f})} We assume that $(k,c_S(R))=(2,5)$ and $4\le v(S)$.
There are two possible Hilbert series for $S/K$ which are consistent with the hypothesis   $(k,c_S(R))=(2,5)$; that is,  $\HS(S/K)$ is equal to 
$1+2t+t^2+t^3$ or $1+2t+2t^2$. Recall the standard-graded polynomial $P'=\kk[X_1,X_2]$, the homogeneous ideal $K'$ of $P'$ and the isomorphism $P'/B'\cong S/K$ of (\ref{P'/B'}). Apply Lemma~\ref{71.5} to see that $\Fitt^1_{P'}(B')$ is equal to the maximal homogeneous ideal $\M'=(X_1,X_2)$ of $P'$. The proof of Lemma~\ref{71.5} identifies $\tsd(P'/B')$ and $\mgd(B')$ for both choices of $\HF(P'/B')$. In each case, the inequality $\tsd(P'/B')<\mgd(B')$ holds; furthermore in each case $\tsd(P'/B')$ is at most $3$. All of the hypotheses of Lemma~\ref{71.9} are in effect. It follows that Theorem~\ref{(0.1)} applies to $R$.

\begin{chunk-no-advance}\label{mar-23}
{\bf  The case $(k,c_S(R))=(3,5)$.} 
Examine the filtration  (\ref{(4)}) in order to see that $\HS(S/K)=1+3t+t^2$.
It follows, in particular, that  \begin{equation}\label{one-dim'}(K+(x_1, x_2, x_3)^2)/K\end{equation} is a one-dimensional $\kk$-vector space and that $(x_1,x_2,x_3)^3\subseteq K$.
Thus, $$K = (y_1,\dots, y_s, q_1,\dots, q_5) + (x_1, x_2, x_3)^3,$$ 
where $q_1,\dots, q_5$ are linearly independent quadratic forms in $(x_1,x_2,x_3)$.
Furthermore,  $y_1,\dots, y_s, q_1,\dots, q_5$ is the beginning of a minimal generating set for $K$.
We show that Hypothesis (\ref{(1)}) of Theorem~\ref{(0.1)} is satisfied. Consider the ten products
\begin{equation}\label{no x_3^3}\{x_iq_j\in S\mid i \in \{1, 2\}\text{ and $1 \le j \le 5$} \},\end{equation}
which is a subset of the ten-dimensional vector space generated by the monomials of degree $3$ in $x_1,x_2,x_3$. (We have used the hypothesis that $3\le v(S)$.)
The monomial $x_3^3$ is not in the vector space spanned by (\ref{no x_3^3}); so the elements of (\ref{no x_3^3}) must be
linearly dependent. This gives rise to a non-trivial relation $\sum_i q_i\ell_i=0$, with each $\ell_i$ a linear form in $x_1$ and $x_2$. The vector space generated by $\ell_1,\dots,\ell_5$ must have dimension more than one because $q\ell=0$ is not possible in $S$ with $q$ a non-zero quadratic form in  $x_1,x_2,x_3$ and $\ell$ a non-zero linear form  in  $x_1,x_2,x_3$. (Again, we have used the hypothesis that $3\le v(S)$.) Thus, $(x_1,x_2)\subseteq \Fitt_S^1(K)$.
A permutation of the subscripts yields that
$x_3$ is  in $\Fitt_S^1(K)$, as
well.  We conclude that Hypothesis  (\ref{(1)}) holds for $R$.

At this point we have identified the beginning, $y_1,\dots,y_s,q_1,\dots,q_5$, of a minimal generating set for $K$ and the beginning, $[S/K]_2$, of the socle of $S/K$. The exact argument we use to finish the proof depends on whether more generators are needed to generate all of $K$ and whether more socle elements are needed to generate all of $\socle(S/K)$.
\end{chunk-no-advance}

\noindent {\bf(\ref{mar-16-17.h})} Assume $(k,c_S(R))=(3,5)$ and $3\le \maxgendeg(K)$. 
We proved in \ref{mar-23} that Hypothesis~(\ref{(1)}) is satisfied. We now consider 
Hypothesis~(\ref{(2)}).
The inclusion $(x_1, x_2, x_3)^3 \subseteq K$ guarantees that $(x_1, x_2, x_3)^2 \subseteq (K :_S \n)$. On the other hand, the hypothesis that some element of $
(x_1, x_2, x_3)^3$ is a minimal generator of $K$ ensures that   $$(x_1, x_2, x_3)^2 \not\subseteq (\n K :_S \n);$$ so, Hypothesis~(\ref{(2)}) is satisfied and Theorem \ref{(0.1)} applies to $R$ in this case.

\medskip\noindent {\bf(\ref{mar-16-17.i})} Assume $(k,c_S(R))=(3,5)$ and
$2\le \dim_{\kk}\socle(S/K)$. We proved in \ref{mar-23} that Hypothesis~(\ref{(1)}) is satisfied. We now consider 
Hypothesis~(\ref{(2)}). 
 In this case, there must be an element $\ell$ of $[S]_1$ which represents an element of the socle of $S/K$. It is clear that $\ell\in (K:_S\n)$. On the other hand, $\ell\n$ contains some of the minimal quadratic generators of $K$; so $\ell\not\in (\n K:_S\n)$. Thus, Hypothesis~(\ref{(2)}) is satisfied and
 Theorem~{\rm\ref{(0.1)}} applies to $R$.

\medskip\noindent {\bf(\ref{mar-16-17.g})} Assume $(k,c_S(R))=(3,5)$,  $\maxgendeg(K)\le 2$,  $\dim_{\kk}\socle(S/K)=1$, and $5\le v(S)$. We use the notation introduced in \ref{mar-23}. The hypotheses ensure that $K$ is minimally generated by $y_1,\dots,y_s,q_1,\dots,q_5$ and the socle of $S/K$ is equal to the one-dimensional vector space $[S/K]_2$. 
 The ideal $B$ of Remark~\ref{B1} is equal to $$(Y_1,\dots,Y_s,Q_1,\dots,Q_5)+A,$$ where  $P$ is the standard-graded polynomial ring $\kk[Y_1,\dots,Y_s,X_1,X_2,X_3]$ and $Q_i$ is the  quadratic form in $X_1,X_2,X_3$ which represents $q_i$. The hypothesis $5\le v(S)$ implies that 
$$A\subseteq \M^{v(S)}\subseteq \M^5\subseteq (Y_1,\dots,Y_s,Q_1,\dots,Q_5),$$ where $\M$ is the maximal homogeneous ideal of $P$; consequently, 
$$B=(Y_1,\dots,Y_s,Q_1,\dots,Q_5).$$Observe that $A\subseteq \M^5$, $S=P/A$ is an Artinian  Gorenstein local ring, and $$[\socle(P/B)]_1=[\socle(S/K)]_1=0.$$ 
Theorem~\ref{0.2B} applies to $P/(A:_PB)$, which according to Remark~\ref{B1}, is equal to $R$. This completes the proof of Theorem~\ref{mar-16-17}
\qed \end{chunk}

\section{The first splitting theorem.}\label{FST}
 
Although not explicitly stated there, Theorem~\ref{(0.1)}  was shown in \cite{SV}.
The proof we give is essentially the dual of the earlier proof.

\begin{lemma}\label{lemma-for-0.1} Let $(S,\n,\kk)$ be a local ring, $J$ be an ideal of $S$, 
and $d:F\to G$ be a homomorphism of free $S$-modules of finite rank. Suppose that 
$I_1(d)\subseteq \n$ and $$(J:_S\n)\cdot  I_1(d)\not\subseteq \n J.$$ 
Then $\kk$ is a direct summand of the $\frac SJ$-module $\ker (d\otimes_S \frac SJ)$.\end{lemma}

\begin{proof}Let $R$ denote $\frac SJ$. 
If $f$ is an element of $F$, then let $\bar f$ denote the image of $f$ in $F\otimes _SR$, which is equal to $\frac {F}{JF}$.

According to the hypothesis, 
there are elements ${\tt s}\in S$, $\theta\in F$, and $\alpha\in G^*$ such that 
${\tt s} \n\subseteq J$ and 
 ${\tt s}\alpha(d(\theta))$ 
is a minimal generator of $J$.
Observe first that   $\overline{{\tt s}\theta}$ is in $\ker(d\otimes_SR)$.
Indeed, 
$$d({\tt s}\theta)\in {\tt s}\im d\subseteq {\tt s} \n G\subseteq JG.$$
Now observe that $\overline{{\tt s}\theta}$ is a minimal generator of $\ker(d\otimes_SR)$. Otherwise, there are elements $s_i\in \n$ and $\theta_i \in F$ with
${\tt s} \theta-\sum_is_i \theta_i\in JF$ and $d\theta_i\in JG$. In this case,
\begin{align*}{\tt s} \alpha(d(\theta))&{}=\alpha\Big(\sum_i s_i d(\theta_i)+\text{an element of $\n J\cdot G$}\Big)\\&{}= \alpha(\text{an element of $\n J\cdot G$})\in \n J,\end{align*}which is a contradiction.

Let $\overline{{\tt s}\theta}, \phi_2,\dots,\phi_m$ be a minimal generating set for 
$\ker(d\otimes_SR)$. Observe that as an $R$-module, $\ker(d\otimes_SR)$ is equal to the direct sum
$$R\overline{{\tt s}\theta}\oplus R(\phi_2,\dots,\phi_m).$$
It is clear that $R\overline{{\tt s}\theta}+ R(\phi_2,\dots,\phi_m)=\ker(d\otimes_SR)$. Furthermore, if $$r \overline{{\tt s}\theta}\in R(\phi_2,\dots,\phi_m),$$ for some $r$ in $R$, then the definition of the $\phi$'s ensures that $r$ is in the maximal ideal of $R$ and  
$r\overline{{\tt s}\theta}=0$ in $F/JF$.

We have shown that $R\overline{{\tt s}\theta}$ is a non-zero direct summand of $\ker(d\otimes_SR)$. The proof is complete because $\n {\tt s}\subseteq J$; and therefore $R\overline{{\tt s}\theta}$ is isomorphic to $\kk$.  \end{proof}

\begin{chunk}\label{beats me}{\bf Proof of Theorem~\ref{(0.1)}.} We apply Lemma~\ref{lemma-for-0.1} to the Artinian Gorenstein local ring $(S, \n, \kk)$. Take $J$ to be the  
ideal $(0:_SK)$, which defines $R$, from the statement of Theorem~\ref{(0.1)}, and take $d$ to be a minimal presentation matrix for $K$ as an $S$-module.  It follows that $I_1(d)=\Fitt_S^1(K)$, see \ref{con1}.
According to Gorenstein duality, see \ref{2C}, the hypothesis $(J:_S\n)\cdot I_1(d)\not\subseteq \n J$ of Lemma~\ref{lemma-for-0.1} is equivalent to
\begingroup\allowdisplaybreaks\begin{align*}
&(0:_S\n J)\not\subseteq \big(0:_S\big((J:_S\n)\cdot I_1(d)\big)\big)\\\iff&
((0:_S J):_S\n)\not\subseteq \big(\big(0:_S(J:_S\n)\big) :_S     I_1(d)\big)\\
\iff&(K:_S\n)\not\subseteq (0:_S(0:_S\n K)):_SI_1(d)\\\iff&
(K:_S\n)\not\subseteq (\n K:_SI_1(d)).
\end{align*}\endgroup
In Theorem~\ref{(0.1)} the hypotheses $\Fitt_S^1(K) = \n$ and 
$K :_S \n \not\subseteq \n K :_S \n$
are both in effect; hence Lemma~\ref{lemma-for-0.1} applies and the proof of Theorem~\ref{(0.1)} is complete. \qed\end{chunk}

\section{The method for the second and third splitting theorems.}\label{68}

The main results in this section are Lemmas~\ref{75.3-ci} and \ref{75.3}. In each result   we identify a particular summand of a second syzygy. These results are used in Section~\ref{65} to prove Theorems~\ref{0.2A} and \ref{0.2B}.

\begin{setup}\label{75.2}Let $P$ be a commutative Noetherian ring,   $A\subseteq B$ be ideals of $P$ with $A\subseteq B^2$, and $\Delta$ be an element of $P$ with 
\begin{enumerate}[\rm(a)]
\item\label{75.2.a}$(A:_PB)=(A,\Delta)$, and
\item\label{75.2.b}$(A:_P\Delta)=B$.
\end{enumerate}
Let 
$S$ denote the ring $S=P/A$, and $R$ denote the ring $R=P/(A,\Delta)$. 
Let $$\text{$\bar{\phantom{x}}$ denote  $-\t_PS$ \quad  and \quad  $\dbar{\phantom{x}}$ denote $-\t_PR$.}$$
We study $BS$ as an $R$-module. 
\end{setup}

\begin{remark} \label{2.4.1} If Setup~\ref{75.2} is in effect and $S$ is  Artinian, Gorenstein, and local, then $BS$ is the canonical module of $R$ because 
$S$ and $R$ have the same Krull dimension 
 and $S$ is a Gorenstein local ring which maps onto $R$. It follows that the canonical module of $R$ is
\begin{equation}\notag 
\textstyle\omega_R=\Hom_S(R,S)=\Hom_{P/A}\left(\frac P{(A,\Delta)},\frac PA\right)=\frac {A:_P\Delta}{A}=\frac BA=BS.\end{equation}\end{remark}

\begin{assumption}\label{75.3.a} In the notation of Setup~{\rm{\ref{75.2}}}, let 
$$B_2\xrightarrow{b_2}B_1\xrightarrow{b_1}B\to 0$$ be a presentation of $B$ by free $P$-modules. Assume that there exists a $P$-module homomorphism $L: B_1\to B_2$ with 
$$b_2\circ L \equiv \Delta\cdot \id_{B_1} \mod AB_1.$$
\end{assumption}

\begin{lemma}\label{75.3-ci} Assume that the notation and hypotheses of {\rm\ref{75.2}} and {\rm\ref{75.3.a}} are in effect. 
 If   $B$  is  generated by a regular  sequence  and  $B\cdot I_1(L)\subseteq (A,\Delta)$, then  the $R$-module $B/B^2$ is a direct summand of the $R$-module $\syz_2^R(BS)$. 
\end{lemma}

\begin{Remark} Our application of Lemma~\ref{75.3} occurs when $B''$ is generated by variables which are not involved in $B'$. One can read the statement of  Lemma~\ref{75.3} with this application in mind; indeed, the statement is   meaningful (and easier to digest) if $B''$, $B_1''$, and $B_2''$ all are zero and $b_1=b_1'$, $b_2=b_2'$ and $L=L'$. \end{Remark}

\begin{lemma}\label{75.3}
Assume that the notation and hypotheses of {\rm\ref{75.2}} and {\rm\ref{75.3.a}} are in effect.  
Suppose that there are free $P$-module  decompositions of $B_1$ and $B_2$ of the form{\rm:}  
$$B_1=B_1'\p B_1''\quad\text{and}\quad B_2=B_2'\p B_2''$$ 
and corresponding decompositions of the homomorphisms $b_1$,  $b_2$, and $L$ of the form{\rm:} 
$$b_1=\bmatrix b_1'&b_1''\endbmatrix,\quad  b_2=\bmatrix b_2'&b_2'''\\0&b_2'',\endbmatrix,\quad \text{and}\quad L=\bmatrix L'&0\\0&L''\endbmatrix,$$
where 
\begin{align*} &b_1':B_1'\to B,&&b_1'':B_1''\to B,&& b_2':B_2'\to B_1',&&b_2'':B_2''\to B_1'',\\
 &b_2''':B_2''\to B_1', &&L':B_1'\to B_2', \quad \text{and}&& L'':B_1''\to B_2'' \end{align*}
are homomorphisms of free $P$-modules. Let $B'$ and $B''$ be the ideals $b_1'(B_1')$ and $b_1''(B_1'')$ of $P$, respectively. 
Suppose further that 
\begin{enumerate} [\rm(a)]
\item\label{75.3.b}
$B\cdot I_1(L')\subseteq (A,\Delta)$, 
\item\label{75.3.c}  
$L'\circ b_2'\equiv \Delta \cdot \id_{B_2'}\mod AB_2'$,
\item\label{75.3.d}
$I_1(L'\circ b_2''')\subseteq (A,\Delta)$, and
\item\label{5.5.d} the complex $B_2'\xrightarrow{b_2'}B_1'\xrightarrow{b_1'}B'\to 0$ is exact.
\end{enumerate}
Then the $R$-module $B'/BB'$ is a direct summand of the $R$-module $\syz_2^R(BS)$.  
\end{lemma}

\noindent Lemmas ~\ref{75.3-ci} and \ref{75.3} are proved in \ref{The proof of Lemma 75.3-ci} and \ref{The proof of Lemma 75.3}, respectively.
 We first identify $\syz_2^R(BS)$ in terms of homomorphisms than can be readily calculated from the given data; see (\ref{75.5.1}). This part of the argument is fairly routine and uses Assumption~\ref{75.3.a} but does not  use the extra hypotheses of either  Lemma~\ref{75.3-ci} or \ref{75.3}.

\begin{further notation}\label{75.4} Assume the notation and hypotheses of Setup~\ref{75.2} and Assumption~\ref{75.3.a}. 
Let
 \begin{equation}\label{75.4.1}\xymatrix{ 
\dots\ar[r]&A_2\ar[r]^{a_2}\ar[d]^{c_2} &A_1\ar[r]^{a_1}\ar[d]^{c_1}&A\ar[r]\ar
@{^(->}[d]&0\\
\dots\ar[r]&B_2\ar[r]^{b_2} &B_1\ar[r]^{b_1}&B\ar[r]&0}\end{equation} be a commutative diagram of $P$-modules with  exact rows and each $A_i$ and $B_i$ free. (Take the bottom row of (\ref{75.4.1}) to be  a Koszul complex on a regular sequence which generates $B$, whenever this is possible.) The hypothesis that $A\subseteq B^2$ guarantees that we may choose $c_1:A_1\to B_1$ with 
\begin{equation}\label{75.4.2} c_1(A_1)\subseteq B B_1.\end{equation} 
Let $\mu:\bigwedge^2 B_1\to B_2$ be a $P$-module homomorphism with $$(b_2\circ \mu)(\beta_1\w \beta_1')= b_1(\beta_1)\cdot \beta_1'-b_1(\beta_1')\cdot \beta_1.$$
 Define
$\phi:A_1\t B_1 \to B_2$ to be  the composition
$$A_1\t B_1 \xrightarrow{c_1\t 1} B_1\t B_1\xrightarrow{\text{natural quotient map}} \textstyle \bigwedge^2 B_1\xrightarrow{\mu}B_2.$$
Let $$ \begin{matrix}
{ B}_3\\\oplus\\ { A}_2\\\oplus \\{ A}_1 \t { B}_1\
\end{matrix}\xrightarrow{\delta_{2,\ell}} \begin{matrix}
{ B}_2\\\oplus\\ { A}_1\end{matrix}, \quad\quad B_1\xrightarrow{\delta_{2,r}} \begin{matrix}
{ B}_2\\\oplus\\ { A}_1\end{matrix}, 
\quad\quad\begin{matrix}
{ B}_3\\\oplus\\ { A}_2\\\oplus \\{ A}_1 \t { B}_1\\\oplus \\B_1
\end{matrix}\xrightarrow{\delta_{2}} \begin{matrix}
{ B}_2\\\oplus\\ { A}_1\end{matrix}, \quad\quad \text{and}\quad\quad  \begin{matrix}
{ B}_2\\\oplus\\ { A}_1\end{matrix}\xrightarrow{\delta_{1}} B_1$$
represent the $P$-module homomorphisms
$$\delta_{2,\ell}=\bmatrix { b}_3&{ c}_2&{ \phi}\\0&-{ a}_2&1\t { b}_1\endbmatrix,\ \ \delta_{2,r}=\bmatrix L\\0\endbmatrix,\ \  \delta_2=\bmatrix \delta_{2,\ell}
&\delta_{2,r}\endbmatrix, \ \ \text{and}\ \  \delta_1=\bmatrix b_2&c_1\endbmatrix
.$$
\end{further notation}

\begin{proposition}\label{75.5}If the notation and hypotheses of {\rm\ref{75.4}} are in effect, then the following statements  hold.
\begin{enumerate}[\rm(a)]
\item\label{75.5.a}The homomorphisms 
$$\begin{matrix}
{\bar B}_3\\\oplus\\ {\bar A}_2\\\oplus \\{\bar A}_1 \t {\bar B}_1\
\end{matrix} \xrightarrow{\ \ \bar{\delta}_{2,\ell}\ \ 
} \begin{matrix}
{\bar B}_2\\\oplus\\ {\bar A}_1\end{matrix} \xrightarrow{\ \ \bar{\delta}_1\ \ 
} 
{\bar B}_1 \xrightarrow{\ \ {\bar b}_1\ \ }BS \to 0 
$$ form an exact complex of $S$-modules.
\item\label{75.5.b}The homomorphisms 
$$\begin{matrix}
\bar{\bar B}_3\\\oplus\\ \bar{\bar A}_2\\\oplus \\\bar{\bar A}_1 \t \bar{\bar B}_1\\\oplus\\\bar{\bar B}_1
\end{matrix} \xrightarrow{\ \ \dbar{\delta_2}\ \ 
} \begin{matrix}
\bar{\bar B}_2\\\oplus\\ \bar{\bar A}_1\end{matrix} \xrightarrow{\ \ \dbar{\delta_1}\ \ } 
\bar{\bar B}_1 \xrightarrow{\ \ \bar{\bar b}_1\ \ }BS \to 0 
$$ form an exact complex of $R$-modules and 
\begin{equation}\label{75.5.1}\syz_2^R(BS)=\ker \dbar{\delta_1}= \im \dbar{\delta_2}.\end{equation}{\rm(}Recall from {\rm\ref{68.1}} that ``$\syz_2^R(BS)$'' is meaningful up to 
formation of direct sum with a $R$-projective module.{\rm)}
\end{enumerate}
\end{proposition}
\begin{proof} In each case the indicated maps form a complex. We show that the complexes are exact. The mapping cone
\begin{equation}\label{75.5.2}\dots \xrightarrow{\bmatrix b_4&c_3\\0&-a_3\endbmatrix}
\begin{matrix} B_3\\\oplus\\ A_2\end{matrix}\xrightarrow{\bmatrix b_3&c_2\\0&-a_2\endbmatrix}
\begin{matrix} B_2\\\oplus\\ A_1\end{matrix}\xrightarrow{\bmatrix b_2&c_1\endbmatrix} B_1\xrightarrow{b_1} \frac BA=BS\to 0
\end{equation} of (\ref{75.4.1}) is a  resolution of $BS$ by free $P$-modules.

\medskip\noindent (\ref{75.5.a}) Apply $-\t_PS$ to (\ref{75.5.2}) to see that \begin{equation}\notag 
\begin{matrix} \sbar {B_2}\\\oplus\\ \sbar {A_1}\end{matrix}\xrightarrow{\sbar{\delta_1}=\bmatrix \sbar {b_2}&\sbar {c_1}\endbmatrix} \sbar {B_1}\xrightarrow{\sbar {b_1}} BS\to 0
\end{equation} is an exact sequence of $S$-modules. We compute $\ker (\sbar{\delta_1})$. Suppose $\theta\in B_2\oplus A_1$ and $\sbar {\delta_1}(\sbar{\theta})=0$ in $\sbar {B_1}$. In this case, $\delta_1(\theta)=\sum_ia_1(\alpha_{1,i})\cdot \beta_{1,i}$ for some $\alpha_{1,i}\in A_1$ and $\beta_{1,i}\in B_1$. Observe that
\begin{align*}(\delta_1\circ\delta_{2,\ell})\left(\bmatrix 0\\0\\\sum\limits_i \alpha_{1,i}\t \beta_{1,i}\endbmatrix\right)&=\bmatrix b_2&c_1\endbmatrix \bmatrix
\sum\limits_i \mu(c_1(\alpha_{1,i})\w \beta_{1,i})\\\sum\limits_i b_1(\beta_{1,i})\cdot \alpha_{1,i}\endbmatrix\\&= \sum_ia_1(\alpha_{1,i})\cdot \beta_{1,i}=
\delta_1(\theta).\end{align*} Thus,
$$\theta -\delta_{2,\ell}\left(\bmatrix 0\\0\\\sum\limits_i \alpha_{1,i}\t \beta_{1,i}\endbmatrix\right)\in \ker (\delta_1).$$
 We see from (\ref{75.5.2}) that 
$$\ker (\delta_1)=\im \left(\bmatrix b_3&c_2\\0&-a_2\endbmatrix\right).$$ It follows that $\sbar{\theta}$ is in the image of $\sbar{\delta_{2,\ell}}$. The completes the proof of (\ref{75.5.a}).

\medskip\noindent (\ref{75.5.b}) The proof of (\ref{75.5.b}) is essentially the same as the proof of (\ref{75.5.a}). The right exactness of tensor product yields the exact sequence
\begin{equation}\notag 
\begin{matrix} \dbar {B_2}\\\oplus\\ \dbar {A_1}\end{matrix}\xrightarrow{\dbar{\delta_1}=\bmatrix \dbar {b_2}&\dbar {c_1}\endbmatrix} \dbar {B_1}\xrightarrow{\dbar {b_1}} BS\to 0.
\end{equation} We compute $\ker (\dbar{\delta_1})$. Suppose $\theta\in B_2\oplus A_1$ and $\dbar {\delta_1}(\dbar{\theta})=0$ in $\dbar {B_1}$. In this case, $\delta_1(\theta)\equiv \Delta \beta_1\mod AB_1$ for some  $\beta_{1}\in B_1$. Apply Assumption~\ref{75.3.a} and observe that
$$(\delta_1\circ \delta_{2,r})\left(\beta_1 \right)
=(b_2\circ L)(\beta_1)\equiv \Delta \beta_1\equiv \delta_1(\theta)\mod AB_1.$$
Thus $\sbar{\theta}-\sbar{\delta_{2,r}}(\sbar{\beta_1})\in \ker(\sbar{\delta_{1}})=\im(\sbar{\delta_{2,\ell}})$, by part (\ref{75.5.a}); and therefore, $\dbar{\theta}$ is in the image of $\dbar{\delta_{2}}$. This completes the proof that the complex of (\ref{75.5.b}) is exact. 

The exact complex of (\ref{75.5.b}) gives rise to the exact sequence 
of $R$-modules
$$0\to \ker \dbar{\delta_1}\to \big(\dbar{B_2}\p \dbar{A_1}\big) \xrightarrow{\dbar{\delta_1}}\dbar{B_1} \xrightarrow{\dbar{b_1}}BS\to 0$$ with $\dbar{B_2}\p \dbar{A_1}$ and $\dbar{B_1}$ free. It follows from Definition~\ref{68.1} that $\ker(\dbar{\delta_1})=\syz_2^R(BS)$. The exactness of the complex of (\ref{75.5.b}) also yields that 
$\ker(\dbar{\delta_1})=\im(\dbar{\delta_2})$. Assertion (\ref{75.5.1}) has been established and the proof of Proposition~\ref{75.5} is complete. \end{proof}

\begin{proposition}\label{82.1} If the notation and hypotheses of Lemma~{\rm\ref{75.3-ci}}  and Notation~{\rm\ref{75.4}} are in effect, then the following statements  hold{\rm:}
 \begin{enumerate}[\rm(a)]
\item\label{82.1.a}$\ker\dbar{L}=\frac{BB_1}{(A:B)B_1}$, and
\item\label{82.1.b}$\im \dbar{L}\cong B/B^2$.\setcounter{nameOfYourChoice}{\value{enumi}}\end{enumerate}
\end{proposition}
\begin{proof} We prove (\ref{82.1.a}) by showing that
\begin{equation}\label{LOOK ME UP}\{\beta_1\in B_1\mid L(\beta_1)\in (A,\Delta)B_2\}=BB_1.\end{equation}
The inclusion ``$\supseteq$'' is guaranteed by the hypothesis $B\cdot I_1(L)\subseteq (A,\Delta)$. We prove ``$\subseteq$''. Let $\beta_1$ be an element of $B_1$ with $L(\beta_1)\in (A,\Delta)B_2$. It follows that there is an element $\beta_2\in B_2$ with $L\beta_1\equiv \Delta \beta_2\mod AB_2$. Apply $b_2$ to both sides and then use Assumption~\ref{75.3.a} to obtain 
$$\Delta \beta_1\equiv (b_2\circ L)(\beta_1)\equiv \Delta b_2(\beta_2)\mod AB_1.$$ It follows that $\Delta (\beta_1-b_2(\beta_2)) \in AB_1$ and
$$ \beta_1-b_2(\beta_2)\in (A:_P\Delta)B_1=BB_1.$$ The hypothesis that $B$ is generated by a regular sequence ensures that 
the bottom row of (\ref{75.4.1}) is a Koszul complex; and therefore,
$b_2(\beta_2)\in BB_1$. We conclude that $\beta_1\in BB_1$; hence, (\ref{82.1.a}) is established. The proof of (\ref{82.1.b}) follows in a routine manner:
$$\im \dbar{L}\cong \frac {\dbar {B}_1}{\ker \dbar{L}}\cong \frac {B_1}{BB_1}\cong \frac{B}{B^2}.$$ The middle isomorphism is due to (\ref{82.1.a}) and the final isomorphism is induced by the surjection $b_1:B_1\to B$.
\end{proof}

\begin{proposition}\label{75.6} If the notation and hypotheses of Lemma~{\rm\ref{75.3}}  and Notation~{\rm\ref{75.4}} are in effect, then the following statements  hold{\rm:} 
 \begin{enumerate}[\rm(a)]
\item\label{75.6.a}$\ker\dbar{L'}=\frac{b_2'(B_2')+B\cdot B_1'}{(A:_PB)B_1'}$, and
\item\label{75.6.b}$\im \dbar{L'}\cong B'/BB'$.\setcounter{nameOfYourChoice}{\value{enumi}}\end{enumerate}
\end{proposition}
\begin{proof}
The proof makes use of the following statements:
\begin{enumerate}[\rm(a)]\setcounter{enumi}{\value{nameOfYourChoice}}
\item\label{75.6.c} $\{\beta_1'\in B_1'\mid L'(\beta_1')\in(A,\Delta)B_2'\}=
b_2'(B_2')+B\cdot B_1'$,  
and 
\item \label{75.6.d} $(A:_PB)B_1'\subseteq b_2'(B_2')+B\cdot B_1'$.
\end{enumerate}

\medskip We  prove (\ref{75.6.c}). Hypotheses (\ref{75.3.c})  and (\ref{75.3.b}) 
of Lemma~\ref{75.3} ensure that 
$$(A,\Delta)B_2'\supseteq L'\big(b_2'(B_2')+B\cdot B_1'\big);$$ and this establishes the inclusion ``$\supseteq$''.  Now, we prove ``$\subseteq$''. Suppose $\beta'_1$ is in $B_1'$ 
with
$$L'(\beta_1')\equiv \Delta \beta_2'\mod AB_2',$$ for some $\beta_2'\in B_2'$. Use Hypothesis \ref{75.3}.(\ref{75.3.c}) to write
$$L'(\beta_1')\equiv (L'\circ b_2') (\beta_2')\mod AB_2'.$$
Apply $b_2'$ to obtain
$$(b_2'\circ L')(\beta_1' - b_2' \beta_2')\equiv 0\mod AB_1'.$$
Employ Assumption~\ref{75.3.a} to see that
$$\Delta(\beta_1' - b_2' \beta_2')\in AB_1';$$
hence, $$\beta_1' - b_2' \beta_2'\in (A:_P\Delta)B_1'=BB_1',$$
by \ref{75.2}.(\ref{75.2.b}), and $\beta_1'\in b_2'(B_2')+ BB_1'$. This completes the proof of 
(\ref{75.6.c}).

\medskip  We prove (\ref{75.6.d}). Recall from \ref{75.2} that $(A:_PB)=(A,\Delta)$ and $A\subseteq B$. Thus, it suffices to consider $\Delta B_1'$. Assumption~\ref{75.3.a} and the decompositions of Lemma~\ref{75.3}  ensure that 
$$\Delta\cdot B_1'\equiv (b_2\circ L)(B_1')=(b_2'\circ L')(B_1')\subseteq b_2'(B_2').$$

\medskip The proof of (\ref{75.6.a}) is now straightforward:
\begin{align*} \ker\dbar{L'}&{}=\frac{\{\beta_1'\in B_1'\mid L'(\beta_1')\in(A,\Delta)B_2'\}+(A:_PB)B_1'}{(A:_PB)B_1'}\\
&{}=\frac{b_2'(B_2')+B\cdot B_1'+(A:_PB)B_1'}{(A:_PB)B_1'},&&\text{by (\ref{75.6.c}),}\\
&{}=\frac{b_2'(B_2')+B\cdot B_1'}{(A:_PB)B_1'},&&\text{by (\ref{75.6.d}).}\end{align*}
The proof of (\ref{75.6.b}) is also straightforward:
$$
\im \dbar{L'}
\cong 
\frac{\dbar{B_1'}}{\ker \dbar{L'}}=\frac {\frac{B_1'}{(A:_PB)B_1'}}{\frac{b_2'(B_2')+B\cdot B_1'}{(A:_PB)B_1'}}
=\frac{B_1'}{b_2'(B_2')+B\cdot B_1'}
\\
\cong 
\frac {B'}{BB'}.$$
The final isomorphism is induced by the surjection $\xymatrix{b_1':B_1'\ar@{->>}[r]& B'}$ and uses Hypothesis~(\ref{5.5.d}) of Lemma~\ref{75.3}
\end{proof}

\begin{chunk}\label{The proof of Lemma 75.3-ci} {\bf The proof of Lemma \ref{75.3-ci}.} We use all of
notation and hypotheses of {\rm\ref{75.2}, \ref{75.3-ci},} and {\rm\ref{75.4}}.
We proved in Proposition~\ref{82.1}.(\ref{82.1.b}) that 
$$\im(\dbar{\delta_{2,r}})\cong B/B^2.$$ 
We complete the proof of Lemma~\ref{75.3-ci} by showing that
$\im(\dbar{\delta_{2,r}})$ is a direct summand of the $R$-module $\syz_2^R(BS)$.
We accomplish this goal by showing that $$\im(\dbar{\delta_{2,\ell}})\oplus \im(\dbar{\delta_{2,r}})=\syz_2^R(BS).$$ We know from (\ref{75.5.1}) that $$\im(\dbar{\delta_{2,\ell}})+ \im(\dbar{\delta_{2,r}})=\syz_2^R(BS).$$ We show that $\im(\dbar{\delta_{2,\ell}})\cap \im(\dbar{\delta_{2,r}})=(0)$. Suppose $$\beta_1\in B_1\quad \text{and}\quad \theta\in B_3\oplus A_2\oplus (A_1\t B_1)$$ with \begin{equation}\label{68.7.1}\delta_{2,r}(\beta_1)-\delta_{2,\ell}(\theta)\quad\text{in}\quad (A,\Delta)B_2\oplus (A,\Delta)A_1.\end{equation} We complete the proof by showing 
that
\begin{equation}\label{68.7.2} L(\beta_1)\in (A,\Delta)B_2.\end{equation}
We are told in (\ref{68.7.1}) that 
$$\bmatrix L(\beta_1)\\0\endbmatrix -\delta_{2,\ell}(\theta)=\bmatrix \Delta\beta_2\\ \Delta \alpha_1 \endbmatrix+\bmatrix \text{an element of $AB_2$}\\\text{an element of $AA_1$} \endbmatrix,$$for some $\beta_2\in B_2$ and $\alpha_1\in A_1$.  
Apply $\delta_1=\bmatrix b_2&c_1\endbmatrix$ to see that 
\begin{equation}\label{68.7.3}b_2L \beta_1-\delta_1\delta_{2,\ell}(\theta)=\Delta b_2 \beta_2+\Delta c_1 \alpha_1+\text{an element of $AB_1$}.\end{equation}We know from \ref{75.5}.(\ref{75.5.a}) that $\delta_1\delta_{2,\ell}(\theta)\in AB_1$. 
The choice $c_1(A)\subseteq BB_1$, which was made in (\ref{75.4.2}),
guarantees that  
$\Delta c_1 \alpha_1\in \Delta BB_1\subseteq AB_1$. Of course, Assumption~\ref{75.3.a} guarantees that $b_2L \beta_1\equiv\Delta \beta_1\mod AB_1$. Thus, (\ref{68.7.3}) yields
$$\Delta(\beta_1-b_2\beta_2)\in AB_1;$$and therefore, $\beta_1-b_2\beta_2\in BB_1$. Thus, $$\beta_1\in b_2(B_2)+BB_1\subseteq BB_1.$$ (The hypothesis that $B$ is generated by a regular sequence ensures that 
the bottom row of (\ref{75.4.1}) is a Koszul complex; and therefore,
$b_2(B_2)\subseteq  BB_1$ and the final inclusion holds.)  Apply the hypothesis $BI_1(L)\subseteq(A,\Delta)$   to conclude $L(\beta_1)\in (A,\Delta)B_2$, which establishes (\ref{68.7.2}) and completes the proof.
\qed\end{chunk}

\begin{chunk}\label{The proof of Lemma 75.3} {\bf The proof of Lemma~\ref{75.3}.} 
We use all of
notation and hypotheses of {\rm\ref{75.2}, \ref{75.3},} and {\rm\ref{75.4}}.
We proved in Proposition~\ref{75.6}.(\ref{75.6.b}) that 
$$\im(\dbar{L'})\cong B'/BB'.$$ 
We complete the proof of Lemma~\ref{75.3} by showing that
$\im(\dbar{L'})$ is a direct summand of the $R$-module $\syz_2^R(BS)$.
We accomplish this goal by showing that $$\Big(\im(\dbar{\delta_{2,\ell}})+ \im \dbar{L''}\Big)\oplus \im(\dbar{L'})=\syz_2^R(BS).$$ We know from (\ref{75.5.1}) that $$\Big(\im(\dbar{\delta_{2,\ell}})+ \im \dbar{L''}\Big)+ \im(\dbar{L'})=\syz_2^R(BS).$$ 
We show that $\Big(\im(\dbar{\delta_{2,\ell}})+ \im \dbar{L''}\Big)\cap \im(\dbar{L'})=(0)$. Suppose $$\beta_1'\in B_1',\quad \beta_1''\in B_1'',\quad \text{and}\quad \theta\in B_3\oplus A_2\oplus (A_1\t B_1)$$ 
with \begin{equation}\label{75.7.1}L'(\beta_1')-L''(\beta_1'')-\delta_{2,\ell}(\theta)\quad\text{in}\quad (A,\Delta)B_2'\oplus (A,\Delta)B_2''\oplus (A,\Delta)A_1.\end{equation} We complete the proof by showing 
that
\begin{equation}\label{75.7.2} L'(\beta_1')\in (A,\Delta)B_2'.\end{equation}
We are told in (\ref{75.7.1}) that 
$$\bmatrix L(\beta_1'-\beta_1'')\\0\endbmatrix -\delta_{2,\ell}(\theta)=\bmatrix \Delta\beta_2\\ \Delta \alpha_1 \endbmatrix+\bmatrix \text{an element of $AB_2$}\\
\text{an element of $AA_1$} \endbmatrix,$$for some $\beta_2\in B_2$,   and $\alpha_1\in A_1$.  
Apply $\delta_1=\bmatrix 
b_2
&c_1\endbmatrix$ to see that 
\begin{equation}\label{75.7.3}(b_2\circ L) (\beta_1'-\beta_1'')
-\delta_1\delta_{2,\ell}(\theta)=\Delta b_2 \beta_2+\Delta c_1 \alpha_1+\text{an element of $AB_1$}.\end{equation}We know from \ref{75.5}.(\ref{75.5.a}) that $\delta_1\delta_{2,\ell}(\theta)\in AB_1$. 
The choice $c_1(A)\subseteq BB_1$, which was made in (\ref{75.4.2}),
guarantees that  
$\Delta c_1 \alpha_1\in \Delta BB_1\subseteq AB_1$. Of course, Assumption~\ref{75.3.a} guarantees that $b_2L \equiv\Delta \mod AB_1$. Thus, (\ref{75.7.3}) yields
$$\Delta(\beta_1'-\beta_1''-b_2\beta_2)\in AB_1;$$and therefore, \begin{equation}\label{newnew}\beta_1'-\beta_1''-b_2\beta_2\in (A:_P\Delta)B_1=BB_1.\end{equation}
Separate (\ref{newnew}) into the components $B_1'\oplus B_1''=B_1$. Recall that 
$b_2$ has the form $$b_2=\bmatrix b_2'&b_2'''\\0&b_2''\endbmatrix;$$ and write $\beta_2=\beta_2'+\beta_2''$ with $\beta_2'\in B_2'$ and $\beta_2''\in B_2''$. Conclude that
$$\bmatrix \beta_1'\\-\beta_1''\endbmatrix -\bmatrix b_2'&b_2'''\\0&b_2''\endbmatrix
\bmatrix \beta_2'\\\beta_2''\endbmatrix= \bmatrix \text{an element of $BB_1'$}\\
\text{an element of $BB_1''$} \endbmatrix.$$ 
It follows that 
$$\beta'_1-b_2'\beta_2'-b_2'''\beta_2'''\in BB_1'.$$
Apply $L'$ and invoke Hypotheses (\ref{75.3.b}), (\ref{75.3.c}), and (\ref{75.3.d}) of Lemma~\ref{75.3} to conclude $L'\beta_1'$ is in $(A,\Delta)B_2'$, 
 which establishes (\ref{75.7.2}) and completes the proof.
\qed\end{chunk}

\section{Complete intersections and five quadratics in three variables.} \label{65}

In this section, we prove Theorems 
\ref{0.2A} and \ref{0.2B} which were promised  in Section~\ref{Intro}.
The proof of Theorem~\ref{0.2A} is based on 
Lemma~\ref{75.3-ci} and is given in {\ref{prove T1}}. 
The  
proof of Theorem~\ref{0.2B}   is given in \ref{proof-of-65.1} after 
 we collect the results that  ensure that the hypotheses of Lemma~\ref{75.3} are satisfied.

\begin{chunk}\label{prove T1} {\bf The proof of Theorem~\ref{0.2A}.}
 We apply Lemma~\ref{75.3-ci} and Remark~\ref{2.4.1}. 
Gorenstein duality guarantees that $(A:_PB)/A$ is a cyclic $P/A$-module. 
Let   $\Delta$ be an element of $P$ with 
$(A:_PB)=(A,\Delta).$ 
Gorenstein duality also guarantees that
$(A:_P\Delta)=B$. 

Let  $\theta_1,\dots, \theta_n$ be a regular sequence in $P$ which generates $B$ and 
let 
$$B_2\xrightarrow{b_2}B_1\xrightarrow{b_1}B\to 0$$ be 
 the beginning of the Koszul complex on 
this generating set.  
We are now in the situation of  Setup~\ref{75.2}.  It remains to  establish the existence of a homomorphism $L:B_1\to B_2$ which satisfies 
\begin{enumerate}[\rm(a)]
\item\label{W.28}$(b_2\circ L) \equiv \Delta\cdot \id_{B_1} \mod AB_1$, and
\item\label{W.29} $B\cdot L(B_1)\subseteq (A,\Delta)B_2$.
\end{enumerate}

Property (\ref{W.29}) is satisfied provided
$$L(B_1)\subseteq \Big((A,\Delta):B)\Big) B_2=\Big((A:B):B)\Big) B_2
=(A:B^2) B_2.$$Consequently, it suffices to prove that 
$$(A:B)B_1\subseteq b_2\Big((A:B^2)B_2\Big)+AB_1.$$

We think of $B_1$ as $P^n$ and of the surjection $B_1\to B$ as the map given by the matrix $[\theta_1,\ \dots,\ \theta_n]$. We show that 
$$\bmatrix (A:B)\\0\\\vdots\\0\endbmatrix\subseteq P^{n}$$ is in $ b_2\Big((A:B^2)B_2\Big)+AB_1$. One completes the argument by symmetry.
Every Koszul relation on $\theta_1,\dots,\theta_n$ is in the image of $b_2$; so, in particular,
each column of 
$$\bmatrix \theta_2& \cdots&  \theta_n\\\hline &-\theta_1I_{n-1}&\endbmatrix$$
 is in the image of $b_2$, where $I_{n-1}$ is the $(n-1)\times (n-1)$ identity matrix.
Let $\Theta$ be an element of $(A:B)$. We show that there exist elements $r_2,\dots, r_n$ in $(A:\theta_1)\cap (A:B^2)$ with $\sum_{i=2}^n r_i \theta_i\equiv \Theta$, mod $A$. In other words, it suffices to show that 
$$\Theta\in \Big((A:\theta_1)\cap (A:B^2)\Big)  (\theta_2,\dots, \theta_n)+A.$$ Let $\Theta$ roam over  $(A:B)$. It suffices to show that 
$$(A:B)\subseteq (A:(\theta_1,B^2))  (\theta_2,\dots, \theta_n)+A.$$
By Gorenstein duality, it suffices to show
\begin{equation}\label{W.30}(\theta_1,B^2):(\theta_2,\dots, \theta_n) \subseteq B.\end{equation} The hypothesis that $B$ is a complete intersection of grade at least two ensures that  (\ref{W.30}) holds and the proof is complete. \qed \end{chunk}

 We turn our attention to the proof  of Theorem~\ref{0.2B}. Some preliminary results are needed. 

\begin{observation}\label{2.2} Let $\kk$ be a field of arbitrary characteristic,   $P'$ be a standard-graded polynomial ring in three variables over $\kk$, and  $B'$ be an ideal in $P'$ which is generated by a $5$-dimensional subspace of $[P']_2$. Assume that  \begin{equation}\label{NDH}
[\socle(P'/B')]_1=0.
\end{equation} Then the ring $P'/B'$ is Gorenstein and has Hilbert series $\HS(P'/B')=1+5t+t^2${\rm;}
furthermore, the ideals $(B')^2$ and $([P']_1)^4$ of $P$ are equal.
\end{observation}

\begin{proof} Let $\phi:[P']_2\to \kk$ be a $\kk$-module homomorphism with $[B']_2=\ker \phi$. If the characteristic of $\kk$ is different than two, then use standard results about symmetric bilinear forms over a field (essentially Gram-Schmidt orthogonalization) to choose a basis $x,y,z$ for $[P']_1$ so that the matrix
\begin{equation}\label{matrix T}T=\bmatrix \phi(x^2)&\phi(xy)&\phi(xz)\\\phi(xy)&\phi(y^2)&\phi(yz)\\\phi(xz)&\phi(yz)&\phi(z^2)\endbmatrix \end{equation} is a diagonal matrix. The non-degeneracy hypothesis (\ref{NDH}) ensures that the entries on the main diagonal of $T$ are units in $\kk$, say $u_1,u_2,u_3$.
The same conclusion holds in characteristic two; but the argument is more subtle. The argument we offer was taken from some expository notes written by Keith Conrad \cite[Exercise~5.4]{C}.

Assume $\kk$ has characteristic two. The non-degeneracy hypothesis (\ref{NDH}) ensures that there is an element $z_1\in [P']_1$ with ${z_1}^2\notin B'$. (Otherwise, if $x_1$, $y_1$, $z_1$ is any basis of  $[P']_1$, then $\phi(y_1z_1) x_1+\phi(x_1z_1)y_1+\phi(x_1y_1)z_1$ represents a non-zero element of $
[\socle(P'/B')]_1$.) If there is an element $y_1$ with $y_1z_1\in B'$ and $y_1^2\notin B'$, then the usual argument maybe used to complete the proof. On the other hand, if every element of $B':_{[P']_1}z_1$ squares to an element of $B'$, then let $x_1,y_1$ be a basis of $B':_{[P']_1}z_1$ with $\phi(x_1y_1)=1$. Consider the basis
$$x=\phi(z_1^2)x_1+z_1, \quad  y=(1+\phi(z_1^2))x_1+y_1+z_1,\quad \text{and}\quad z=x_1+y_1+z_1$$ for $[P']_1$.
Observe that the matrix (\ref{matrix T}) is a diagonal matrix. In this case also,
we take the entries on the main diagonal of $T$ to be the  units  $u_1,u_2,u_3$ of  $\kk$.

It is now clear that \begin{equation}\label{B}
B'=(xy,xz,yz,u_2x^2-u_1y^2,u_3x^2-u_1z^2).\end{equation} It is also clear that the row vector of signed maximal order Pfaffians  of the alternating matrix 
\begin{equation}\label{b_2}b_2'=\bmatrix
 0   &  u_3x   & -u_1u_3y& -z & 0    \\
 -u_3x & 0     & u_1u_3z  &0  & y    \\
 u_1u_3y& -u_1u_3z& 0      &u_3x& -u_2x \\
 z    & 0     & -u_3x   &0  & 0    \\
 0    & -y   &  u_2x    &0 &  0    \endbmatrix\end{equation}
is \begin{equation}\label{2.2.3}b_1'=[u_3xy,u_2xz,yz,u_2u_3x^2-u_1u_3y^2,u_3^2x^2-u_1u_3z^2].\end{equation} Observe that the image of $b_1'$ is equal to $B'$ and that $B'$ has grade $3$. It follows that
\begin{equation}\label{6.2.6}0\to P'(-5)\xrightarrow{{b_1'}^{\rm T}} P'(-3)^5 \xrightarrow{b_2'}P'(-2)^5 \xrightarrow{b_1'}
P'\end{equation} is a minimal homogeneous resolution of $P'/B'$; hence, $P'/B'$ is a 
 Gorenstein ring 
with Hilbert series $1+5t+t^2$. One can check the assertion ${B'}^2=([P']_1)^4$ by hand or one can read the minimal homogeneous resolution
$$0\to P'(-6)^{10}\to P'(-5)^{24}\to P'(-4)^{15}\to P'$$
of $P'/{B'}^2$ from \cite[page 36]{KU} in order to conclude that the socle degree of $P'/{B'}^2$ is three; thus, $([P']_1)^4\subseteq {B'}^2$. The inclusion ${B'}^2\subseteq ([P']_1)^4$ is obvious.
\end{proof}

\begin{lemma}\label{2.3}
 Let $\kk$ be a field of arbitrary characteristic,   $P=\kk[X_1,X_2,X_3,Y_1,\dots,Y_s]$ be a standard-graded polynomial ring over $\kk$  in $3+s$  variables for some nonnegative integer $s$,  $B'$ be an ideal in $P$ which is generated by five linearly independent quadratic forms in the variables $X_1,X_2,X_3$,
$B''$ be the ideal $(Y_1,\dots,Y_s)$ of $P$,  $B$ be the ideal $B'+B''$ of $P$,
and  $A\subseteq ([P]_1)^5$ be a homogeneous ideal of $P$ with $P/A$ an Artinian Gorenstein local ring.  
Assume    $[\socle(P/B)]_1=0$.
Then there is a presentation
\begin{equation}\label{pres}B_2\xrightarrow{b_2}B_1\xrightarrow{b_1}B\to 0\end{equation}
 of $B$ by 
free $P$-modules,
a 
$P$-module homomorphism $L: B_1\to B_2$, and direct sum decompositions
$$B_1=B_1'\p B_1'',\ \  B_2=B_2'\p B_2'',\ \  
b_1=\bmatrix b_1'&b_1''\endbmatrix,\ \   b_2=\bmatrix b_2'&b_2'''\\0&b_2'',\endbmatrix,\ \  \text{and}\ \  L=\bmatrix L'&0\\0&L''\endbmatrix,$$
which satisfy Assumption~{\rm\ref{75.3.a}} and conditions {\rm(\ref{75.3.b})}, {\rm(\ref{75.3.c})}, {\rm(\ref{75.3.d})}, and {\rm(\ref{5.5.d})}
of Lemma~{\rm\ref{75.3}}. \end{lemma}

\begin{proof} Let $P'$ be the polynomial ring $\kk[X_1,X_2,X_3]$. View $P'$ as a subring of $P$. Notice that $P/B=P'/(P'\cap B')$; so we may apply Observation~\ref{2.2} and pick a basis $x,y,z$ of $[P']_1$ so that the generators of $(P'\cap B')$ are given in (\ref{B}) and the minimal homogeneous resolution of $P'/(P'\cap B')$ by free $P'$-modules is given in (\ref{6.2.6}) with $b_1'$ and $b_2'$ given in (\ref{2.2.3}) and (\ref{b_2}), respectively. 

Apply the functor $P\t_{P'}-$ to (\ref{6.2.6}) to obtain the minimal homogeneous resolution 
$$\mathbb B':\quad 0\to B_3'=P(-5)\xrightarrow{{b_1'}^{\rm T}}B_2'=P(-3)^5 \xrightarrow{b_2'}
B_1'=P(-2)^5\xrightarrow{b_1'} B_0'=P$$ of $P/B'$. (This establishes Hypothesis~(\ref{5.5.d}) of 
Lemma~\ref{75.3}.) Let $B_1''$ be a free $P$-module of rank $s$,  $b_1'':B_1''\to P$ be the homomorphism $$P(-1)^s\xrightarrow{\bmatrix Y_1&\cdots&Y_s\endbmatrix} P,$$ and $\bigwedge^{\bullet}B_1''$ be the Koszul complex associated to $b_1''$. The minimal homogeneous resolution of $P/B$ is $\mathbb B'\t_P \bigwedge^{\bullet}B_1''$. In this language, the presentation (\ref{pres}) of $B$ by free $P$-modules is 
$$\begin{matrix} B_2'\t_P \bigwedge^0B_1''\\\p\\B_1'\t \bigwedge^1B_1''
\\\p\\B_0'\t \bigwedge^2B_1''\end{matrix} \xrightarrow{\ \ b_2\ \ } 
\begin{matrix} B_1'\t_P \bigwedge^0B_1''\\\p\\B_0'\t \bigwedge^1B_1''
\end{matrix}\xrightarrow{\ \ b_1=\bmatrix b_1'&b_1''\endbmatrix\ \ } B=B'+B''\to 0,$$where 
$$b_2=\left[\begin{array}{c|cc} b_2'&-1\t b_1''&0\\\hline0&b_1'\t 1&1\t b_1''\end{array}\right].$$
To complete the promised decomposition of  (\ref{pres}) let
\begin{align}\label{promised}B_2''= \begin{matrix} B_1'\t \bigwedge^1B_1''
\\\p\\B_0'\t \bigwedge^2B_1'',\end{matrix}\qquad 
&b_2'''=\bmatrix -1\t b_1''&0\endbmatrix\quad \text{
and},\\\notag  &b_2''=\bmatrix b_1'\t 1&1\t b_1''\endbmatrix.\end{align}

The homogeneous ideals $A$ and $B$ satisfy $A\subseteq B$ and both ideals define Gorenstein rings. Gorenstein duality ensures the existence of a homogeneous element $\Delta$ of $P$ with
\begin{equation}\label{2.3.0}(A:_PB)=(A,\Delta) \quad\text{and}\quad (A:_P\Delta)=B.\end{equation}
Let $S$ denote $P/A$, and $\bar{\phantom{x}}$ denote $-\t_PS$.
 Let $\binom{x,y,z,Y_1,\dots,Y_s}{i}$ denote the set of monomials in $P$ of degree $i$ in $x,y,z,Y_1,\dots,Y_s$. The hypothesis $A\subseteq ([P]_1)^5$  
guarantees that $$\textstyle\{\sbar{m}\mid m\in \binom{x,y,z,Y_1,\dots,Y_s}{i}\}$$ is a basis for $[S]_i$ for $0\le i\le 4$. 
Let $\s$ denote the socle degree of $S$ and let $\alpha_1$ be an element of $[P]_\s$ with the property  that $\sbar{\alpha_1}$ is a basis for the socle $[S]_\s$ of $S$. For each integer $i$, with $0\le i\le \s$, the  multiplication map \begin{equation}\label{pp}[S]_i\times [S]_{\s-i}\to [S]_\s\end{equation} is a perfect pairing. For $1\le i\le 3$, select $\{\alpha_m\in [P]_{\s-i}\mid m\in \binom{x,y,z,Y_1,\dots,Y_s}{i}\}$
so that 
  $$\textstyle \{\sbar{\alpha_m}\mid m\in \binom{x,y,z,Y_1,\dots,Y_s}{i}\}$$ is a basis for $[S]_{\s-i}$ which is dual to $\{\sbar{m}\mid m\in \binom{x,y,z,Y_1,\dots,Y_s}{i}\}$ in the sense that
if $m,m'\in \binom{x,y,z,Y_1,\dots,Y_s}{i}$, then 
$$\sbar {m'}\, \sbar{\alpha_{m}}=\begin{cases}\sbar{\alpha_1},&\text{if $m=m'$, and }\\0,&\text{if $m\neq m'$}.\end{cases}$$ It follows from (\ref{pp}) that if 
$m\in\binom{x,y,z,Y_1,\dots,Y_s}{i}$ and $m'\in\binom{x,y,z,Y_1,\dots,Y_s}{i'}$ for some $i$ and $i'$ with $0\le i,i'\le 3$, then
$$\sbar {m'}\, \sbar{\alpha_{m}}=
\begin{cases} \sbar{\alpha_{m''}},&\text{if $m'm''=m$, and }\\0,&\text{if $m'$ does not divide $m$.}\end{cases}$$

The generators of $B$ are $Y_1,\dots,Y_s$, together with the five quadratics  given in (\ref{B}). Use the perfect pairing (\ref{pp}), with $i=2$, to observe that 
$$[A:_PB]_{\s-2}=[A]_{\s-2}+\kk(u_1\alpha_{x^2}+u_2\alpha_{y^2}+u_3\alpha_{z^2});$$
hence, we may take the $\Delta$ of (\ref{2.3.0}) to be  \begin{equation}\label{2.3.3}\Delta=u_1\alpha_{x^2}+u_2
\alpha_{y^2}+u_3\alpha_{z^2}.\end{equation}

\begin{table}
\begin{center}
Let $L'$ be the matrix
$$\left[\begin{array}{c|c|c|c|c}
0&-\frac{u_1}{u_3}\alpha_{x^3}&0&u_2\alpha_{y^2z}&0\\
&&&+u_3\alpha_{z^3}&\\
\hline
0&0&0&-\frac{u_1u_2}{u_3}\alpha_{x^2y}&-u_1\alpha_{x^2y}\\
&&&&-u_2\alpha_{y^3}\\&&&&-u_3\alpha_{yz^2}\\
\hline
     -\frac{u_2}{u_1u_3}\alpha_{y^3}&\frac1{u_1}\alpha_{z^3}&0&-\frac{u_1}{u_3}\alpha_{x^3}&0\\\hline
     -u_1\alpha_{x^2z}&\frac{u_2^2}{u_3}\alpha_{y^3}&\frac{u_2}{u_3}\alpha_{xy^2}&-u_1u_2\alpha_{xyz}&-u_1u_3\alpha_{xyz}\\
-u_3\alpha_{z^3}&&&&\\
\hline
     -\frac{u_1u_3}{u_2}\alpha_{x^2z}&u_2\alpha_{y^3}&-\frac{u_1}{u_2}\alpha_{x^3}&0&0\\
-\frac{u_3^2}{u_2}\alpha_{z^3}&&-\frac{u_3}{u_2}\alpha_{xz^2}&&
\end{array}  \right]. $$
\caption{{ The matrix $L'$ from the proof of Lemma~\ref{2.3}.}}\label{table}
\end{center}\end{table}
Let $L'$ be the $5\times 5$ matrix of Table~\ref{table} with entries in $P$. Recall the matrix $b_2'$ from (\ref{b_2}).
A straightforward calculation shows that $b_2'  L'$ and $L'b_2'$ are both congruent to $\Delta  I_5$, modulo $A$. The conclusion 
$b_2'  L'\equiv \Delta  I$ establishes half of Assumption~\ref{75.3.a}. (The other half is studied at (\ref{apr26}).) The conclusion $L'b_2'\equiv \Delta  I$ demonstrates that   condition  Lemma~\ref{75.3}.(\ref{75.3.c})  holds. 
Observe that
\begin{equation}\label{B''}I_1(b_2''')\cdot S\subseteq B''\cdot S\subseteq \ann_S I_1(L');\end{equation}
and therefore condition   Lemma~\ref{75.3}.(\ref{75.3.d})  holds. (The second inclusion is easily read from the definition of $L'$.)
To prove condition  Lemma~\ref{75.3}.(\ref{75.3.b}), 
it suffices to show that $B^2\cdot I_1(L')\cdot S=0$. On the other hand, 
\begin{align*}B^2\cdot I_1(L')\cdot S&{}\subseteq 
({B'}^2+B'')I_1(L')\cdot S\\&{}= {B'}^2I_1(L')\cdot S&&\text{see (\ref{B''})}\\
&{}\subseteq 
[S]_4\cdot [S]_{\s-3}\cdot S\subseteq [S]_{\s+1}\cdot S=0.\end{align*}

\begin{chunk-no-advance}\label{apr26} To complete the proof, we must establish the rest of Assumption~\ref{75.3.a}; that is, we must  exhibit a $P$-module homomorphism $L'':B_1''\to B_2''$ with \begin{equation}\label{final goal}b_2\bmatrix 0\\L''\endbmatrix\equiv\bmatrix 0\\\Delta \cdot \id_{B_1''}\endbmatrix \mod A B_1.\end{equation}
The next calculation is the main step in that direction. \end{chunk-no-advance}
\begin{claim-no-advance}\label{Apr-10} The element $\Delta$ of $P$ is  in the ideal $B'\cdot(A:_PB'')+A$.\end{claim-no-advance}

\medskip\noindent{\bf Proof of Claim~\ref{Apr-10}.} 
By Gorenstein duality, \ref{2C}, it suffices to show that
$$A:_P(B'\cdot(A:_PB'')+A)\subseteq A:_P(A,\Delta). $$
The module on the right is $B$. The module on the left is
$$(A:_P(A:_PB'')):_PB'=(A,B''):_PB'.$$ To prove Claim~\ref{Apr-10} it suffices to show that 
\begin{equation}\label{STS}(A,B''):_PB' \subseteq B.\end{equation}
Let $p$  be a homogeneous element  in $(A,B''):_PB'$. Write $p=p'+p''$ for homogeneous elements $p'$ and $p''$, with $p'$ in the subring $\kk[x,y,z]$ of $P$ and $p''$ in the ideal $B''$ of $P$. Of course, $p''$ is in $(A,B''):_PB'$ and $p''\in B$; consequently, $p'$ is in $(A,B''):_PB'$
and it suffices to prove that $p'\in B$. We are told that 
$p'B'\subseteq (A,B'')$. There are non-zero elements of
$$[p'B'\cap \kk[x,y,z]]_{\deg p'+2}.$$ On the other hand, the ideal $A$ is contained in $([P]_1)^5$, by hypothesis, so $$[(A,B'')\cap \kk[x,y,z]]_i=0\quad\text{for $i\le 4$.}$$Therefore, $5\le \deg p'+2$; hence $3\le \deg p'$ and $p'\in B'\subseteq B$; and this completes the proof of Claim~\ref{Apr-10}. (Recall that the Hilbert series of $P'/B'$ is given in Observation~\ref{2.2}.)

Claim~\ref{Apr-10} guarantees that there is an element $c$ in $B_1'$ with
\begin{equation}\label{from-Apr-10}b_1'(c)\equiv \Delta \mod A \quad \text{and}\quad B''I_1(c)\subseteq A.\end{equation}
Recall from (\ref{promised}) that $B_2''=(B_1'\t B_1'')\p \bigwedge^2 B_1''$. 
Define $L'':B_1''\to B_2''$ by 
$$L''(\beta_1'')=c\t \beta_1''\in B_1'\t B_1''\subseteq B_2'', \qquad\text{for $\beta_1''\in B_1''$}.$$ Observe that
$$b_2\bmatrix 0\\L''\endbmatrix(\beta_1'')=\bmatrix b_2'&-1\t b_1''&0\\0&b_1'\t 1&1\t b_1''\endbmatrix\bmatrix 0\\c\t \beta_1''\\0\endbmatrix 
=\bmatrix -b_1''(\beta_1'')\cdot c\\b_1'(c)\cdot \beta_1''\endbmatrix \equiv \bmatrix 0\\\Delta\cdot \beta_1''\endbmatrix, $$
mod $A$. The final congruence is due to (\ref{from-Apr-10}). This completes the proof of (\ref{final goal}) and also the proof of Lemma~\ref{2.3}. \end{proof}

\begin{remark} We used Macaulay2, \cite{GS}, to solve a system of 250 non-homogeneous linear equations in 300 unknowns in order to produce the matrix $L'$ of Table~\ref{table}. The matrix $L'$ is $5\times 5$; so, it has $25$ entries, and each entry is a linear combination of the ten elements $\{\alpha_m\mid m\in \binom{x,y,z}3\}$. Altogether, $25\times 10$ coefficients must be determined. Each entry of $b_2'L'$ and each entry of $L'b_2'$ is a linear combination of the six elements $\{\alpha_m\mid m\in \binom{x,y,z}2\}$. Altogether $(25+25)\times 6$ equations must be solved. It is amusing to notice that, according to the computer, the vector space of solutions of the corresponding homogeneous equations has dimension $86$, over $\mathbb Q$ and also over $\mathbb Z/(2)$; so, there are many choices for $L'$. The fact that we used the computer to identify one choice for $L'$ is essentially irrelevant because it is easy to check by hand that the matrix of Table~\ref{table} does satisfy $\sbar{b_2'L'}=\sbar{L'b_2'}=\sbar{\Delta} I$. 
\end{remark}

\begin{chunk}\label{proof-of-65.1}{\bf The proof of Theorem~\ref{0.2B}.} 
 Apply Lemma~\ref{2.3} to see that all of the hypotheses of Lemma~\ref{75.3} hold. Lemma~\ref{75.3} yields the result. 
\qed
\end{chunk}

\section{Test modules}\label{TM}
In Theorem~\ref{v2-test-module} and Proposition~\ref{7.2} we prove that many modules are proj-test modules in the sense of Definition~\ref{proj-test}. The modules ``$B'/BB'$'' of Theorem~\ref{0.2B} are covered by Theorem~\ref{v2-test-module} and Remark~\ref{plus one}. The modules ``$B/B^2$'' of Theorem~\ref{0.2A} are covered in  Proposition~\ref{7.2} and Remark~\ref{plus two}. We state and prove the main result in the paper as Theorem~\ref{Main Result}.

\begin{theorem}\label{v2-test-module} Let $(R,\m,\kk)$ be a local Noetherian ring with $3\le v(R)$.
If 
 $T$ 
is a finitely generated  $R$-module with $\m^2 T = 0$, then $T$ is a proj-test module
 for $R$, in the sense of Definition~{\rm \ref{proj-test}}. 
\end{theorem}
\begin{proof} Let $x_1,\dots,x_n$ be a minimal generating set for $\m$.  
The assumption  about $v(R)$ guarantees that the $\binom{n+1}2$ monomials of degree $2$ in the symbols $x_1,\dots,x_n$ represent a basis for the $\kk$-vector space $\m^2/\m^{3}$; see \ref{local}.(\ref{v(R)}).

Fix finitely generated $R$-modules $T$ and $M$ with $\m^2T=0$ and 
$\Tor_+^R(T,M)=0$. 
 Let
$$\mathcal F :\quad  
\cdots \xrightarrow{d_2}R^{b_1}\xrightarrow{d_1}R^{b_0}$$
be 
the minimal resolution of $M$. 
The matrix  $d_i$ is well-defined for all non-negative integers $i$; if necessary, $d_i$ is allowed to be the zero matrix. Furthermore, each  entry of each $d_i$ is in $\m$.
Write
$$d_i =\sum\limits_{j=1}^nD_{i,j}x_j,$$
where each $D_{i,j}$ is a $b_{i-1}\times b_{i}$ matrix 
with entries in $R$.
Use $d_id_{i+1} = 0$ and the hypothesis $3\le v(R)$ to see that  
\begin{align}\label{v2-(82.1)} D_{i,j} D_{i+1,j}&{}\equiv 0 \mod \m, &&\text{for all $j$, and}\\
\label{v2-(82.2)} D_{i,j_1}D_{i+1,j_2} + D_{i,j_2}D_{i+1,j_1} &{}\equiv 0 \mod \m, &&\text{for all $j_1 \neq j_2$}.\end{align}
 The hypotheses 
  $\m^2T=0$ and 
$\Tor_+^R(T,M)=0$ guarantee that 
$$(\m T)\t 
 R^{b_{i}} \subseteq \ker(T\t 
 d_{i}) = \im(T \t
 d_{i+1})$$

Fix a positive index $i$. If   $\pmb u$ is an arbitrary vector  in  $(\m T)^{b_{i}}$, then
\begin{equation}\label{v2-(3)} \pmb u =\sum _{j_0=1}^n
D_{i+1,j_0}x_{j_0}\pmb u_{j_0},\end{equation}
for some $\pmb u_1,\ldots, \pmb u_n \in T^{b_{i+1}}$.
Each $x_{j_0}\pmb u_{j_0}$ is in $(\m T)^{b_{i+1}}$;  therefore 
$$x_{j_0}\pmb u_{j_0} \in \im(T\t 
 d_{i+2})\quad\text{and}$$
\begin{equation}\label{v2-combine me}x_{j_0}\pmb u_{j_0} =
\sum_{
j_1=1}^n
D_{i+2,j_1}x_{j_1}\pmb u_{j_1,j_0},\end{equation}
for some $\pmb u_{j_1,j_0} \in  T^{b_{i+2}}$.
Combine  (\ref{v2-(3)}) and (\ref{v2-combine me}) to obtain
$$\pmb u =\sum^n_
{j_0=1}\sum^n_
{j_1=1}
D_{i+1,j_0}D_{i+2,j_1}x_{j_1}\pmb u_{j_1,j_0}.$$
In its turn,  $x_{j_1}\pmb u_{j_1,j_0} \in (\m T)^{b_{i+2}} \subseteq \im(T\t
d_{i+3})$. 
When the above procedure is iterated $k$ times, one obtains
\begin{equation}\label{v2-(4)} \pmb u =\sum_{j_0=1}^n\sum_{j_1=1}^n\cdots \sum_{j_k=1}^n 
D_{i+1,j_0}D_{i+2,j_1} \cdots D_{i+k+1,j_k}x_{j_k}\pmb u_{j_k,\ldots, j_1,j_0} ,\end{equation}
for some  $\pmb u_{j_k,\ldots, j_1,j_0}\in T^{b_{i+k+1}}$ .

Observe that if $n<k$, then 
\begin{equation}\label{v2-Apr9}D_{i+1,j_0}D_{i+2,j_1} \cdots D_{i+k+1,j_k}x_{j_k}\equiv 0,\mod \m\end{equation} for each such product which appears in (\ref{v2-(4)}). Indeed there must be  two $j$'s
which are the same, say $j_{i_1} = j_{i_2}$ for some $i_1$ and $i_2$ with $0 \leq i_1 < i_2 \leq k$.
If $i_2 = i_1 + 1$, then  equation (\ref{v2-(82.1)}) shows that (\ref{v2-Apr9}) holds. If not, use equation (\ref{v2-(82.2)})
to bring the equal $j$'s  closer together.

It now follows that $(\m T)^{b_i} = 0$, and therefore $\m T = 0$. In other words,
$T$ is a vector space $\bigoplus \kk$. The hypothesis that $\Tor_1(T,M)$ is zero now implies that $\Tor_1(\kk,M)=0$; and therefore, $M$ is a free $R$-module.
\end{proof}
\begin{remark}\label{plus one} Theorem~\ref{v2-test-module} may be applied to 
the module $B'/BB'$ of Theorem~\ref{0.2B}.
 Indeed,  $B'/BB'$  is annihilated by $(X_1,X_2,X_3)^2$ because
$$(X_1,X_2,X_3)^2B'\subseteq (X_1,X_2,X_3)^2(X_1,X_2,X_3)^2=(X_1,X_2,X_3)^4={B'}^2\subseteq BB'.$$The right most equality is established in Observation~\ref{2.2}. \end{remark} 

\begin{proposition}\label{7.2}Let $P$ be a regular ring, 
$\M$ be a maximal ideal of $P$,
 $I\subseteq \M$ be an   ideal of $P$ which is generated by a regular sequence, and $R$ be $P/J$, where $J$ is an $\M$-primary ideal of $P$  
and $v(R)$  is sufficiently large. Then $T=R/IR$ is a proj-test module for $R$  in the sense of Definition~{\rm \ref{proj-test}}.  \end{proposition}

\begin{proof} It does no harm to localize $P$ at $\M$ 
 and  to inflate the residue field of $P_{\M}$ in order to assume that the inflated residue field is infinite; see, for example \cite[Ch. 0, 10.3.1]{EGAIII}. 
We alter the notation to make it less cumbersome.
In the new notation, 
$(P,\M,\kk)$ is 
a regular local ring with $\kk$  infinite, $I$ 
is  an   ideal of $P$ which is generated by a regular sequence, and $R$ 
is an Artinian  quotient of $P$ with $v(R)$ sufficiently large.

The Artin-Rees Lemma guarantees the existence of an integer $n_0$ such that $$I\cap \M^{n_0}\subseteq I\M.$$ In this proof, we insist that $n_0<v(R)$.

 Let $M$ be a finitely generated $R$-module with $\Tor_i^R(M,T)=0$ for all positive $i$. The ring $R$ and the $R$-module $M$ both have depth zero; consequently, according to the Auslander-Buchsbaum formula, to show that $M$ is a free $R$-module, it suffices to show that $M$ has finite projective dimension as an $R$-module.  We assume that $M$ has infinite projective dimension and we draw a contradiction.

Let $$F:\quad
\dots \xrightarrow{d_3} F_2\xrightarrow{d_2} F_1\xrightarrow{d_1} F_0$$ be a collection of $P$-module homomorphisms of free $P$-modules with the property that
${F\t_PR}$ is the minimal $R$-resolution of $M$. The hypothesis about the vanishing of $\Tor^R_+(M,T)$ guarantees that  
$(F\t_PR)\t_RT=F\t_PT$ is the minimal resolution of $M\t_RT$ by free $T$-modules. 

The ring $T$ is a complete intersection, the ring $P$ is regular local, and the residue field $\kk$ of $P$ is infinite;  hence, the Eisenbud-Peeva \cite{EP} theorems apply to $F\t_PT$. In particular, if $i$ is sufficiently large, then the syzygy module $\im(d_i\t_PT)$ in $F\t_PT$ is a Higher Matrix Factorization module \cite[Theorem 1.3.1]{EP}; and therefore, when suitable bases are chosen in $F$, some of the entries of the product matrix $d_{i+1}d_{i+2}$ are minimal generators of $I$. The hypothesis that $n_0<v(R)$ guarantees that every minimal generator of $I$ represents a non-zero element of $R$. On the other hand, $F\t_PR$ is a complex and $d_{i+1}d_{i+2}\t_P R=0$. We have reached a contradiction
$$0\neq d_{i+1}d_{i+2}\t_P R=0;$$
and this completes the proof. 
\end{proof}

\begin{remark}\label{plus two} Proposition~\ref{7.2} may be applied to 
the module $B/B^2$ of Theorem~\ref{0.2A}.
 Indeed,  the fact that the ideal $B$ of $P$ is generated by a regular sequence ensures that $B/B^2$ is a free $P/B$ module of rank $c$, where $c$ is the minimal number of generators of $B$; hence, as an $R$-module,
$$\frac B{B^2}=\frac B{B^2}\t_P R=\left(\frac PB\right)^c\t_P R=\left(\frac R{BR}\right)^c.$$ 
 \end{remark} 

\medskip Theorem~\ref{Main Result} is the main result of the paper. Recall that if  $R=S/J$ is a quotient of the ring $S$, then $c_S(R)$ is the length of $J$ as an $S$-module. Also, recall  the parameter $v(S)$ from \ref{local}.(\ref{v(R)}). 
\begin{theorem}\label{Main Result} Let $(S,\n,\kk)$ be an equicharacteristic  Artinian Gorenstein local ring with embedding dimension at least two, 
 and  $R$ be the ring $R=S/J$ for some proper non-zero 
 ideal   $J$ of $S$. Assume that either
\begin{itemize}
\item $1\le c_{S}(R)\le 4$, or
\item $c_S(R)=5$ and $S$ is a standard-graded algebra over a field $[S]_0$.
\end{itemize}
If the parameter  $v(S)$ 
 is sufficiently large, then $R$ is $G$-regular.
\end{theorem}

\begin{proof}Apply Theorem~\ref{mar-16-17} to see that one of the Theorems \ref{(0.1)}, \ref{0.2A}, or  \ref{0.2B} applies to $R$. These theorems ensure that either $\kk$, ``$B/B^2$'', or ``$B'/BB'$'' is a direct summand of $\syz_2^R(\omega_R)$. The module $\kk$ is the prototype of a proj-test module. It is shown in Theorem~\ref{v2-test-module} (together with Remark~\ref{plus one}) and Proposition~\ref{7.2} (together with Remark~\ref{plus two}) that ``$B/B^2$'' and ``$B'/BB'$'' are proj-test modules. Apply Observation~\ref{2.13} to complete the proof. \end{proof}

\end{document}